\documentclass[12pt]{amsart}
\usepackage{amscd,amsxtra,amssymb, bbm}

\usepackage[dvips]{graphics}
\usepackage{epsfig,graphicx}

\input xy
\xyoption{all}
\xyoption{arc}
\SelectTips{cm}{10}

\newtheorem{theorem}{Theorem}[section]
\newtheorem{corollary}[theorem]{Corollary}
\newtheorem{lemma}[theorem]{Lemma}
\newtheorem{proposition}[theorem]{Proposition}

\theoremstyle{definition}
\newtheorem{definition}[theorem]{Definition}
\newtheorem{remark}[theorem]{Remark}

\newtheorem{example}[theorem]{Example}

\newcommand{\can}{\mathrm{can}}

\newcommand{\Tr}{\mathrm{Tr}}

\DeclareMathOperator{\res}{\mathsf{res}}
\DeclareMathOperator{\ev}{\mathsf{ev}}
\DeclareMathOperator{\Sol}{\mathsf{Sol}}

\DeclareMathOperator{\Coh}{\mathsf{Coh}}

\DeclareMathOperator{\VB}{\mathsf{VB}}

\DeclareMathOperator{\Pic}{\mathsf{Pic}}

\DeclareMathOperator{\Hom}{\mathsf{Hom}}

\DeclareMathOperator{\Ext}{\mathsf{Ext}}
\DeclareMathOperator{\Lin}{\mathsf{Lin}}
\DeclareMathOperator{\GL}{\mathsf{GL}}

\DeclareMathOperator{\Mat}{\mathsf{Mat}}

\newcommand{\CC}{\mathbb{C}}
\newcommand{\ZZ}{\mathbb{Z}}

\newcommand{\kA}{\mathcal{A}}

\newcommand{\kE}{\mathcal{E}}
\newcommand{\kF}{\mathcal{F}}
\newcommand{\kG}{\mathcal{G}}

\newcommand{\kO}{\mathcal{O}}
\newcommand{\kL}{\mathcal{L}}
\newcommand{\kP}{\mathcal{P}}

\newcommand{\kM}{\mathcal{M}}

\newcommand{\kV}{\mathcal{V}}

\newcommand{\mm}{\mathsf{m}}

\newcommand{\lar}{\longrightarrow}

\begin{document}

\title[Semi--stable vector bundles and the associative Yang--Baxter equation]{Semi--stable vector bundles on
elliptic curves  and  the associative Yang--Baxter equation}

\author{Igor Burban}
\address{
Mathematisches Institut,
Universit\"at Bonn,
Endenicher Allee 60, 53115 Bonn,
Germany
}
\email{burban@math.uni-bonn.de}

\author{Thilo Henrich}
\address{
Mathematisches Institut,
Universit\"at Bonn,
Endenicher Allee 60, 53115 Bonn,
Germany
}
\email{henrich@math.uni-bonn.de}

\begin{abstract}
In this paper we  study unitary  solutions of the associative Yang--Baxter equation (AYBE)
with spectral parameters.
We show that to  each point $\tau$ from the upper half-plane
  and an invertible $(n\times n)$ matrix $B$ with complex coefficients one can attach a solution of AYBE with values in $\Mat_{n \times n}(\CC) \otimes \Mat_{n \times n}(\CC)$,
depending holomorphically on  $\tau$ and  $B$.
 Moreover, we compute some of these solutions explicitly.
\end{abstract}

\maketitle

\section{Introduction}

\noindent
Let $\tau \in \CC$ be such that ${\mathsf{Im}}(\tau) >0$, $q = \exp(\pi i \tau)$ and
$$
\theta(z) = \theta_1(z|\tau) = 2 q^{\frac{1}{4}} \sum\limits_{n=0}^\infty (-1)^n q^{n(n+1)}
\sin\bigl((2n+1)\pi z\bigr)
$$
be the first theta-function of Jacobi. Consider the following meromorphic function
$\sigma(u, x) = \frac{\displaystyle \theta'(0) \theta(u+x)}{\displaystyle \theta(u) \theta(x)}$
 introduced by  Kronecker, see for example \cite[Chapter VIII]{Weil}. It is well-known that $\sigma(u, x)$  satisfies the
 celebrated  Fay's identity
\begin{equation}\label{E:Fay}
\sigma(u, x) \sigma(u+v, y) = \sigma(u+v, x+y) \sigma(-v, x) + \sigma(v, y) \sigma(u, x+y).
\end{equation}
Note that the Kronecker function also satisfies the equality $\sigma(-u, -x) = - \sigma(u, x)$.
As it was shown by Polishchuk \cite[Theorem 5]{Polishchuk1}, up
to a certain equivalence relation the Kronecker function and its trigonometric and rational degenerations
$\cot(u) + \cot(x)$ and $\frac{\displaystyle 1}{\displaystyle u} +
\frac{\displaystyle 1}{\displaystyle x}$
are the only solutions of the functional equation (\ref{E:Fay}).

In this paper we study solutions of the matrix-valued generalization of the Fay's identity (\ref{E:Fay}).
Namely, we are interested in  meromorphic functions
$r: (\CC^2, 0) \rightarrow A \otimes A$,  satisfying
the following equality:
\begin{equation}\label{E:AYBE}
r^{12}(u, x)  r^{23}(u+v, y) = r^{13}(u+v, x+y) r^{12}(-v, x) +  r^{23}(v, y)  r^{13}(u, x+y),
\end{equation}
where $A = \Mat_{n \times n}(\CC)$.
The upper indices in  (\ref{E:AYBE})   indicate various embeddings of
$A \otimes A$ into $A \otimes A \otimes A$.
For example,
the function $r^{13}$ is defined as the composition
$$
r^{13}: \mathbb{C}^2 \stackrel{r}\lar A \otimes A
\stackrel{\rho_{13}}\lar A \otimes A  \otimes A,
$$
where $\rho_{13}(x\otimes y) = x \otimes \mathbbm{1} \otimes y$. The two other
maps $r^{12}$ and $r^{23}$ have a similar meaning. The equation (\ref{E:AYBE}), called
\emph{associative Yang--Baxter equation},
was introduced by Polishchuk  \cite{Polishchuk1}. Its theory was further developed
by Burban and Kreu\ss{}ler in \cite{BK4}.

The first version of the associative Yang--Baxter equation (without spectral parameters)
appeared in a paper of Fomin and Kirillov \cite{FominKirillov}. Later, it arose in
a work  of Aguiar in the framework of the deformation theory of Hopf algebras
\cite{Aguiar}. A special version of the equation (\ref{E:AYBE}) was also considered by Odesskii and Sokolov \cite{OdeSok}.

In what follows,  we shall be interested in  \emph{unitary} solutions of the associative Yang--Baxter equation,   i.e.~in solutions of (\ref{E:AYBE}) satisfying an additional  identity
$
r(-u, -x) = - \rho\bigl(r(u, x)\bigr),
$
where $\rho: A \otimes A \rightarrow A \otimes A$ is the automorphism
given by the rule
$\rho(a \otimes b) = b \otimes a$ for all $a, b \in A$.  In this case,  the function $r(u, x)$ automatically
satisfies the ``dual equation''
\begin{equation}\label{E:AYBEdual}
r^{23}(v, y) r^{12}(u+v, x) =
r^{12}(u, x) r^{13}(v, x+y) +  r^{13}(u+v, x+y) r^{23}(-u, y),
\end{equation}
see for example \cite[Lemma 2.7]{BK4}. The unitary solutions of
(\ref{E:AYBE})  having the Laurent expansion
with respect to the first variable of the form
\begin{equation}\label{E:Laurant}
r(u, x) = \frac{\mathbbm{1} \otimes \mathbbm{1}}{u} + r_0(x) + u r_1(x) + \dots
\end{equation}
were studied by Polishchuk \cite{Polishchuk1, Polishchuk2} as well as by  Burban and Kreu\ss{}ler
\cite{BK4}. Such solutions  are
 closely related with the classical and quantum  Yang--Baxter equations, see \cite{Polishchuk1, Polishchuk2, BK4} for more details.

In this paper we construct  non-degenerate unitary solutions of
(\ref{E:AYBE}) \emph{not} satisfying the residue condition (\ref{E:Laurant}). Moreover, we
get solutions
having \emph{higher order} poles with respect to the first spectral parameter $u$.
It turns our that they can be frequently expressed
via the Kronecker function $\sigma(u, x)$ and its derivatives with respect to the first variable. For example, we show that the elliptic function
$$
r(u, x) = \sigma(u, x) \bigl(e_{11} \otimes e_{11} + e_{22} \otimes e_{22} +
e_{12} \otimes e_{21} + e_{21} \otimes e_{12}\bigr) + \sigma'(u, x)
\bigl(e_{12} \otimes e_{11} - e_{12} \otimes e_{22} - $$
$$
e_{11} \otimes e_{12} + e_{22} \otimes e_{12}\bigr) - \sigma''(u, x) e_{12} \otimes e_{12}
$$
is a unitary solution of (\ref{E:AYBE}) for $A = \Mat_{2 \times 2}(\CC)$, where derivatives
of $\sigma(u, x)$  are taken
with respect to the first variable.

The  study of solutions of the associative Yang--Baxter equation   is also motivated by an observation
of Kirillov \cite{Kirillov} stating that any unitary solution of (\ref{E:AYBE}) determines a certain
  family of commuting
first order  differential operators  and hence, a very interesting  quantum integrable systems, see also \cite[Proposition 2.9]{BK4}. In particular, one can attach to any
unitary solution of $(\ref{E:AYBE})$
a second-order differential operator of Calogero-Moser type, generalizing
the construction of Buchstaber, Felder and Veselov \cite{BFV}. We hope
that our approach to the construction of these operators via the theory of vector bundles on genus one curves will be helpful to clarify their spectral properties.
 On the other hand, our explicit solutions of (\ref{E:AYBE})  provide new identities for the higher derivatives of the
Kronecker function  $\sigma(u, x)$.

The main result of our paper is the following. We fix a complex parameter $\tau \in \CC$ such
that $\mathsf{Im}(\tau) > 0$ and an invertible matrix $B \in \GL_n(\CC)$.
Let $\Lambda = \ZZ + \tau \ZZ$ be the corresponding lattice in $\CC$ and
 $\mathfrak{G}(B)
= \left\{\lambda_1, \dots, \lambda_n\right\}$ the spectrum of $B$. We denote by  $\Sigma = \Sigma_B$ 
the lattice $\left\{\lambda - \lambda'
\, \big| \, \exp(2 \pi i\lambda), \exp(2 \pi i\lambda') \in \mathfrak{G}(B) \right\} + \Lambda \subset \CC$.  Then we attach to the pair
$(B, \tau)$
a meromorphic tensor-valued function $$r_B = r_B(v, y): \CC \times \CC \lar
\Mat_{n \times n}(\CC) \otimes \Mat_{n \times n}(\CC)$$ having the following properties:

\begin{enumerate}
\item The function $r_B$ is a non-degenerate unitary solution of
(\ref{E:AYBE}).
\item Moreover, $r_B$ depends \emph{analytically} on the entries of the matrix $B$
and is holomorphic on $(\CC \setminus \Sigma)  \times (\CC \setminus  \Lambda)$.
\item Let $S \in \GL_n(\CC)$ and $A = S^{-1} B S$. Then we have:
$$r_A(v, y) = (S^{-1} \otimes S^{-1}) r_{B}(v, y) (S \otimes S)$$
 i.~e.~$r_A$ and $r_B$ are \emph{gauge equivalent} in the sense of
\cite[Definition 2.5]{BK4}.
\item If $B = \mathsf{diag}\bigl(\exp(2\pi i \lambda_1), \dots, \exp(2\pi i \lambda_n)\bigr)$
for some $\lambda_1, \dots, \lambda_n \in \CC$
then the corresponding solution $r_B$ is given by the  following formula:
$$
r_B(v, y) =
\sum\limits_{k, \,  l = 1}^n \sigma(v - \lambda_{kl}, y) \, e_{l,    k} \otimes e_{k,    l},
$$
where $\lambda_{kl} = \lambda_k - \lambda_l$ and $\sigma(u, x)$ is the Kronecker function.
\item If $B = J_n(1)$ is the Jordan block of size $n \times n$ with  eigenvalue one then
$$
r_B(v, y) = \sum\limits_{\begin{smallmatrix}0\leq k\leq n-1\\
0\leq l\leq n-1\end{smallmatrix}} \nabla_{kl} \bigl(\sigma (v, y)\bigr) \sum\limits_{\begin{smallmatrix}1\leq i\leq n-l\\
1\leq j\leq n-k\end{smallmatrix}}  e_{i,   j+k} \otimes e_{j,  i+l},
$$
where $\nabla_{kl}$ are certain differential operators described in Definition \ref{D:def-of-nabla}.
\end{enumerate}

\noindent
The core of  our method
is the computation of certain \emph{triple Massey products} in the derived category $D^b\bigl(\Coh(E)\bigr)$, where $E = \CC/\Lambda$ is the complex torus
corresponding to the lattice $\Lambda$.

\medskip
\noindent
\emph{Acknowledgement}. This  work  was supported  by the DFG grant Bu--1866/1--2.

\section{Brief description of the main construction}\label{S:BundlesandAYBE}

In this section we present  an  algorithm attaching to a pair
$(B, \tau) \in \GL_n(\CC) \times \mathbb{H}$, where
  $\mathbb{H} \subset \CC$ is the upper half--plane,
  a non-degenerate unitary solution of the associative Yang--Baxter equation
(\ref{E:AYBE}) with values in
$\Mat_{n \times n}(\CC) \otimes \Mat_{n \times n}(\CC)$.
The explanation of this method  as well as proofs will be given in the
next section.

In what follows,  we denote  $\Lambda = \ZZ + \tau \ZZ$. Let
 $\mathfrak{G}(B)
= \bigl\{\lambda_1, \dots, \lambda_n\bigr\}$  be  the spectrum of $B$ and  $\Sigma = \Sigma_B \subset \CC$ be
the lattice $\left\{\lambda - \lambda'
\, \big| \, \exp(2 \pi i\lambda), \exp(2 \pi i\lambda') \in \mathfrak{G}(B) \right\} + \Lambda$. We construct
the tensor--valued function $$r_B: (\CC \setminus \Sigma) \times (\CC \setminus  \Lambda)
\lar \Mat_{n \times n}(\CC) \otimes \Mat_{n \times n}(\CC)$$ in the following way.

\begin{itemize}
\item For any  $v  \in \CC$ consider the function
$$e(z) = e(z, v, \tau) := - \exp\bigl(-2 \pi i (z + v + \tau)\bigr).$$
\item Let $\mathsf{Sol} = \mathsf{Sol}_{B, \,  v, \, \tau}$ be the following  complex vector space:
$$
\mathsf{Sol}
 =
\left\{\Phi: \CC \lar \Mat_{n \times n}(\CC)
\left|
\begin{array}{l}
\Phi \mbox{\textrm{\quad is holomorphic}} \\
\Phi(z+1) = \Phi(z) \\
\Phi(z+\tau) B = e(z) B \Phi(z)
\end{array}
\right.
  \right\}.
$$
\item For any $y \in \CC \setminus \Lambda$ consider the \emph{evaluation map}
$\ev_y: \mathsf{Sol} \rightarrow \Mat_{n \times n}(\CC)$ given by the formula
$\ev_y(\Phi) = \frac{1}{\bar{\theta}(y + \frac{\tau+1}{2})} \Phi(y)$, where
$$
\bar{\theta}(y) = \theta_3(y|\tau) = 1 + 2 \sum\limits_{n=0}^\infty q^{n^2}
\cos\bigl(2 \pi n y\bigr)
$$
is  the \emph{third} Jacobian  theta--function with  $q = \exp(\pi i \tau)$.
Next, consider the \emph{residue  map} $\res_0: \mathsf{Sol} \rightarrow \Mat_{n \times n}(\CC)$ given by the formula $\res_0(\Phi) = \Phi(0)$.
\end{itemize}

\begin{proposition}\label{P:preparationary}
For any $v \in \CC$ and $B \in \GL_n(\CC)$ the vector space $\Sol_{B, \, v, \,  \tau}$  has dimension
$n^2$. Moreover, if $v \notin \Sigma$ then the linear  map
$\res_0: \Sol_{B,\,  v,\,  \tau} \rightarrow \Mat_{n \times n}(\CC)$ is an isomorphism.
\end{proposition}

\noindent
For a proof of this Proposition, see  Corollary \ref{C:dim-of-Sol}, Theorem
\ref{T:keythm} and Remark \ref{R:resmap-is-iso}.

\medskip
\noindent
Next, we continue the construction of the tensor valued function $r_{B}$.

\begin{itemize}
\item For any pair $(v , y) \in (\CC \setminus \Sigma) \times (\CC \setminus  \Lambda)$
 consider the linear
map $\tilde{r}_{B}(v, y)$ given by the following commutative diagram:
\begin{equation}\label{E:keydiagram}
\begin{array}{c}
\xymatrix@C-C@R+3mm
{
\Mat_{n \times n}(\CC)  \ar[rr]^{\tilde{r}_{B}(v, y)}  & & \Mat_{n \times n}(\CC) \\
& \Sol_{B, \, v,  \, \tau} \ar[ul]^{\res_0}
\ar[ur]_{\ev_y}& \\
}
\end{array}
\end{equation}
In other words, $\tilde{r}_{B}(v, y) := \ev_y \circ \res_0^{-1}$.
\item Let $r_{B}(v, y) \in \Mat_{n \times n}(\CC) \otimes \Mat_{n \times n}(\CC) $ be the tensor
corresponding to the linear map $\tilde{r}_{B}(v, y)$ via the canonical map
of vector spaces
$$
\mathsf{can}:
\Mat_{n \times n}(\CC) \otimes \Mat_{n \times n}(\CC) \lar \Hom_\CC\bigl(\Mat_{n \times n}(\CC), \Mat_{n \times n}(\CC) \bigr)
$$
sending a simple tensor $X \otimes Y$ to the linear map $Z \mapsto \Tr(XZ) Y$.
\end{itemize}

\noindent
The following theorem is the main result of our article.

\begin{theorem}\label{T:main}
Let $(B, \tau) \in \GL_{n}(\CC) \times \mathbb{H}$.
\begin{enumerate}
\item
 The function
$(\CC \setminus \Sigma) \times (\CC \setminus  \Lambda) \rightarrow
\Mat_{n \times n}(\CC) \otimes \Mat_{n \times n}(\CC)$, assigning  to a pair
$(v, y)$ the tensor $r_B(v, y)$ constructed above, is a non--degenerate
holomorphic unitary solution
of the associative Yang--Baxter equation (\ref{E:AYBE}). Moreover, this function
 is meromorphic on $\CC \times \CC$.
\item Let  $S \in \GL_n(\CC)$ and $A = S^{-1} B S$. Then for any $(v, y) \in
(\CC \setminus \Sigma) \times (\CC \setminus  \Lambda)$ we have the following equality:
    $$
    r_A(v, y) = (S^{-1} \otimes S^{-1}) r_B(v, y) (S \otimes S).
    $$
In particular, the solutions $r_A$ and $r_B$ are gauge--equivalent in the sense
of \cite[Definition 2.5]{BK4}.
\end{enumerate}
\end{theorem}

\noindent
The key idea  of the proof of this  theorem is to interpret the linear
morphism $\tilde{r}_{B}(v, y)$ from the diagram (\ref{E:keydiagram}) as a certain triple
Massey product in the derived category $D^b\bigl(\Coh(E)\bigr)$, where $E = \CC/\Lambda$.
The fact that ${r}_{B}(v, y)$ satisfies the equation (\ref{E:AYBE})
is a translation of the $A_\infty$--constraint $m_3 \circ (m_3 \otimes
\mathbbm{1} \otimes \mathbbm{1} +
\mathbbm{1} \otimes m_3 \otimes \mathbbm{1} +
\mathbbm{1} \otimes \mathbbm{1} \otimes m_3) = 0$. For further details, see Theorem \ref{T:thmcles},
Proposition \ref{P:transl-invariance} and Proposition \ref{P:gauge-transf}.

\section{Proof of the main theorem}
In this section we explain the algorithm of the construction of solutions of the associative Yang--Baxter equation (\ref{E:AYBE}) stated in Section \ref{S:BundlesandAYBE} and
prove    Theorem \ref{T:main}.

\subsection{Vector bundles on a one-dimensional complex torus}
Let $\tau \in \mathbb{H}$, $\Lambda = \ZZ + \tau \ZZ \subset \CC$ and $E = \CC/\Lambda$ be the corresponding complex torus. In this subsection we recall
the  basic techniques for dealing  with  holomorphic vector bundles on $E$.

\begin{definition}
Let $A: \CC \rightarrow \GL_n(\CC)$ be a holomorphic function satisfying the condition
$A(z+1) = A(z)$ for all $z \in \CC$. Such a function $A$, called \emph{automorphy factor},
 defines the following topological space
$\kE(A) := \CC \times \CC^n/\sim$, where $(z, v) \sim (z+1, v) \sim (z + \tau, A(z)v)$.
Note that we have a Cartesian diagram
of complex manifolds
$$
\xymatrix
{
\CC \times \CC^n \ar[r] \ar[d]_{\mathrm{pr}_1} & \kE(A)
\ar[d]\\
\CC \ar[r]^\pi  & E
}
$$
and $\kE(A)$ is a vector bundle of rank $n$ on the torus $E$.
\end{definition}

\begin{remark}\label{R:TrivialAutomFact}
 Let $\pi: \CC \rightarrow \CC/\Lambda = E$ be the quotient map.
Another way to define the locally free sheaf $\kE(A)$ is the following.

The open subsets $U\subset E$ for which all connected components of
$\pi^{-1}(U)$ map isomorphically to $U$, form a basis of the topology of $E$.
For such $U$, we let $U_{0}$ be a connected component of $\pi^{-1}(U)$ and
denote $U_\gamma = \gamma + U_{0}$ for all $\gamma\in\Lambda$.
Then
$\pi_*\kO_\CC^n\bigl(U\bigr) = \prod_{\gamma \in \Lambda} \kO^{n}_\CC(U_\gamma)$
and we define
$$
\kE(A)\bigl(U\bigr) :=
\left\{
(F_\gamma(z))_{\gamma \in \Lambda} \in \pi_*(\kO_\CC^n)\bigl(U\bigr)
\left|
\begin{array}{l}
F_{\gamma + 1}(z +1) = F_\gamma(z) \\
F_{\gamma + \tau}(z + \tau)  = A(z)F_\gamma(z)  \\
\end{array}\!\!\!
\right.
  \right\}.
$$
In this way we get  an embedding
$\mm_A: \kE(A)\subset \pi_* \kO_\CC^n$ as well as  a  trivialization $\gamma_A$ of $\pi^*\bigl(\kE(A)\bigr)$
given by the composition
$
\pi^* \kE(A) \xrightarrow{\pi^*(\mm_A)}
\pi^*\pi_* \kO_\CC^n \stackrel{\can}\lar \kO_\CC^n.
$

\end{remark}


\noindent
The following classical result is due to A.~Weil.

\begin{theorem}\label{T:AutomFactors}
Let $E = \CC/\Lambda$ be a one-dimensional complex torus.
\begin{enumerate}
\item For any holomorphic rank $n$ vector bundle $\kE$ on the torus $E$ there exists
an automorphy factor  $A: \CC \rightarrow \GL_n(\CC)$  such that $\kE \cong \kE(A)$.
\item For any  automorphy factors  $A:  \CC \rightarrow \GL_n(\CC)$ and
$B: \CC \rightarrow \GL_m(\CC)$  we have:
$$
\Hom\bigl(\kE(A), \kE(B)\bigr) \cong \Sol_{A, B} :=
\left\{\Phi: \CC \to \Mat_{m \times n}(\CC)
\left|
\begin{array}{l}
\Phi \mbox{\rm{\quad is holomorphic}} \\
\Phi(z+1) = \Phi(z) \\
\Phi(z+\tau)A(z) = B(z) \Phi(z)
\end{array}
\right.
  \right\}
$$
and $\kE(A) \otimes \kE(B) \cong \kE(A \otimes B)$.
\end{enumerate}
\end{theorem}

\begin{proof}
This result is a corollary  of the monoidal equivalence of the  category of $\Lambda$--equivariant
holomorpic vector bundles on $\CC$ and holomorphic vector bundles on
the quotient torus $E = \CC/\Lambda$. See \cite{BL} or  \cite{Iena} for a detailed proof.
\end{proof}

\begin{corollary}\label{C:AutomFactors}
{
For any pair of automorphy factors  $A, S: \CC \rightarrow \GL_n(\CC)$  we have an isomorphism of vector bundles
$
\kE(A) \cong \kE(B),
$
where $B(z) = S(z + \tau)^{-1} A(z) S(z)$. In particular,  we have an isomorphism
$\kE(A) \cong \kE(\widehat{A})$, where $\widehat{A}(z) =
\exp(2 \pi i \tau) A(z)$.
}
\end{corollary}

\noindent
In  the next step, we need an explicit description of the indecomposable semi-stable vector
bundles on $E$ of degree zero.

\begin{theorem}\label{T:BundlesTorus} Let $E = \CC/\Lambda$ be a complex torus.
\begin{enumerate}
\item The map $\CC \rightarrow \Pic(E)$ assigning to $\lambda \in \CC$ the line bundle
$\kL_\lambda := \kE\bigl(\exp(2 \pi i \lambda)\bigr)$ yields a bijection between the points
of $E$ and the isomorphy classes of degree zero line bundles on $E$.
\item
For any $m\ge 1$ let $$
J_m = J_m(1) = 
\left(
\begin{array}{ccccc}
1 & 1 & 0 & \dots & 0 \\
0 & 1 & 1 &  \dots & 0 \\
\vdots & \vdots & \ddots & \ddots & \vdots \\
0 & 0 & \dots & 1 & 1 \\
0 & 0 & \dots & 0 & 1
\end{array}
\right) \in \GL_m(\CC).
$$
Then $\kE(J_m)$ is isomorphic to the Atiyah bundle $\kA_m$  defined
as follows. For $m = 1$ we set  $\kA_1 = \kO$ and for $m \ge 2$ the vector
bundle
$\kA_{m}$ is recursively defined by the following property: it is the unique (up to an isomorphism)  vector bundle
 occurring as the middle term of a non-split  short exact sequence
$$
0 \lar \kA_{m-1} \lar \kA_m \lar \kO \lar 0.
$$
\item Let $B \in \GL_n(\CC)$ and $J = J_{m_1}(\mu_1) \oplus \dots \oplus J_{m_t}(\mu_t)$
be the Jordan normal form of $B$ with  $\mu_l = \exp(2 \pi i \lambda_l)$ for some
$\lambda_l \in \CC$, $ 1 \le l \le t$. Then we have:
$$
\kE(B) \cong (\kL_{\lambda_1} \otimes \kA_{m_1}) \oplus \dots \oplus (\kL_{\lambda_t} \otimes
\kA_{m_t}).
$$
In particular, $\kE(B)$ is a semi--stable vector bundle of degree zero on the torus $E$, whose
Jordan--H\"older quotients are $\kL_{\lambda_1}, \dots, \kL_{\lambda_t}$. Moreover, for
any semi-stable vector bundle $\kE$ of rank $n$ and degree zero on the torus $E$
there exists
a matrix $B \in \GL_n(\CC)$ such that $\kE \cong \kE(B)$.
\end{enumerate}
\end{theorem}

\begin{proof} A proof of the first two statements can for instance be found in
\cite[Section 8.1]{BK4} or in \cite{Iena}.  To show the third one observe  that by Corollary
\ref{C:AutomFactors} we have an isomorphism $\kE(B) \cong \kE(J)$. Since for
any $\lambda \in \CC$ and $m \in \mathbb{N}$ we have an isomorphism
$\kE\bigl(J_m(\lambda)\bigr) \cong \kL_\lambda \otimes \kA_m$, we have:
$\kE(J) \cong (\kL_{\lambda_1} \otimes \kA_{m_1}) \oplus \dots \oplus (\kL_{\lambda_t} \otimes
\kA_{m_t})$. Hence, the result follows from Atiyah's classification
of vector bundles on $E$ \cite{Atiyah}.
\end{proof}

\begin{corollary}\label{C:nomorphdegzerobundles}
Let $A \in \GL_n(\CC)$ and
$\mathfrak{G}(A) = \bigl\{\exp(2 \pi i \lambda_1), \dots, \exp(2 \pi i \lambda_n) \bigr\}$ be its spectrum,
$B \in \GL_m(\CC)$ and $\mathfrak{G}(B) = \bigl\{\exp(2 \pi i \mu_1), \dots, \exp(2 \pi i \mu_m) \bigr\}$
be its spectrum. Assume that $\lambda_k - \mu_l \notin \Lambda$ for all
$1 \le k \le n$ and $1 \le l \le m$. Then we have:
$$
\Hom\bigl(\kE(A), \kE(B)\bigr) = 0 = \Ext^1\bigl(\kE(A), \kE(B)\bigr).
$$
\end{corollary}

\begin{proof}
The assumption on the eigenvalues of $A$ and $B$ implies that the degree zero semi-stable
vector bundles $\kE(A)$ and $\kE(B)$ have no common Jordan--H\"older quotients. From this
fact it follows that
$$\Hom\bigl(\kE(A), \kE(B)\bigr) = 0 = \Hom\bigl(\kE(B), \kE(A)\bigr) \cong
\Ext^1\bigl(\kE(A), \kE(B)\bigr)^*,
$$
where the last isomorphism is given by the Serre duality.
\end{proof}

\begin{lemma}\label{L:LBdegone}
Let $\varphi(z) = \exp(-\pi i \tau - 2 \pi i z)$, $x \in \CC$ and $[x]$ be the corresponding
divisor of degree one on $E$. Then we have an isomorphism:
$$
\kO_E\bigl([x]\bigr) \cong \kE\bigl(\varphi(z + \frac{\tau +1}{2}-x)\bigr).
$$
\end{lemma}

\begin{proof}
A proof of this result can be for instance found in \cite[Section 8.1]{BK4}.
\end{proof}

\subsection{Residue and evaluation morphisms} Let $\Omega_E$ denote the sheaf
of regular differential one forms on the torus $E$. Then we have
an isomorphism $\kO_E \cong \Omega_E$ given by a  nowhere vanishing
differential form, e.g.~by $\omega = dz$. For any $x \in E$
consider the canonical short exact sequence
\begin{equation}\label{E:residue}
0 \to \Omega_E \to \Omega_E(x) \stackrel{\underline{\res}_x}\lar \CC_x \to 0.
\end{equation}
Let $\kF$ and $\kG$ be a pair of vector bundles on $E$. We identify the line bundles
 $\Omega_E$ and $\kO_E$ using the differential form $\omega$,
tensor the sequence (\ref{E:residue}) with $\kG$ and then apply
the functor $\Hom(\kF,\,-\,)$. As a result, we obtain a long exact sequence
\begin{equation}\label{E:todefineresidue}
0 \lar
\Hom(\kF, \kG) \lar \Hom\bigl(\kF, \kG(x)\bigr) \lar
\Hom\bigl(\kF, \kG \otimes \CC_x\bigr) \lar \Ext^1(\kF, \kG).
\end{equation}

\begin{definition}
The linear map $\res^{\kF, \,  \kG}_x(\omega): \Hom\bigl(\kF, \kG(x)\bigr)
\rightarrow \Lin\bigl(\kF\big|_x, \; \kG\big|_x\bigr)$ is   the composition
of the following canonical morphisms
$$\Hom\bigl(\kF, \kG(x)\bigr) \lar
\Hom\bigl(\kF, \kG \otimes \CC_x\bigr)
\lar
\Lin\bigl(\kF\big|_x, \; \kG\big|_x\bigr),
$$
where the first map comes from the long exact sequence
(\ref{E:todefineresidue})
and the second one is a canonical isomorphism. The morphism   $\res^{\kF, \,  \kG}_x(\omega)$ is called \emph{residue map}.
\end{definition}

\noindent
The following lemma is a straightforward corollary of the definition and
the long exact sequence (\ref{E:todefineresidue}).

\begin{lemma}\label{L:ResidueisIso}
Let $\kF$ and $\kG$ be a pair of vector bundles on $E$ such that $\Hom(\kF, \kG) = 0
= \Ext^1(\kF, \kG)$. Then for any $x \in E$ the residue map $\res^{\kF, \, \kG}_x(\omega)$
is an isomorphism.
\end{lemma}

\begin{definition}
Let $\kF$ and $\kG$ be a pair of vector bundles on $E$ and $x, y \in E$ be a pair of
\emph{distinct} points. Then the linear map $\ev^{\kF, \, \kG(x)}_y$
defined as the composition of the following canonical morphisms
$$
\Hom\bigl(\kF, \kG(x)\bigr) \lar \Hom\bigl(\kF \otimes \CC_y, \kG(x) \otimes \CC_y\bigr)
\stackrel{\cong}\lar  \Lin\bigl(\kF\big|_y, \; \kG\big|_y\bigr)
$$
is called \emph{evaluation map}.
\end{definition}

\begin{lemma}
Let $\kF$ and $\kG$ be a pair of vector bundles on $E$ and $x, y \in E$ be a pair of distinct
points such that $\Hom\bigl(\kF(y), \kG(x)\bigr) = 0 = \Ext^1\bigl(\kF(y), \kG(x)\bigr)$.
Then the evaluation map $\ev^{\kF, \, \kG(x)}_y$ is an isomorphism.
\end{lemma}

\begin{proof}
Consider the short exact sequence
\begin{equation}\label{E:evaluat}
0 \lar \kO(-y) \lar \kO \stackrel{\underline{\ev}_y}\lar \CC_y \lar 0.
\end{equation}
It induces a short exact sequence of coherent sheaves
$$
0 \lar  \kG(x-y) \lar  \kG(x)
\xrightarrow{\mathbbm{1} \otimes \underline{\ev}_y}
\kG(x) \otimes \CC_y \to 0.
$$
Using the vanishing $\Hom\bigl(\kF, \kG(x-y)\bigr) = 0 = \Ext\bigl(\kF, \kG(x-y)\bigr)$, we get
an isomorphism  $\Hom\bigl(\kF, \kG(x)\bigr) \rightarrow \Hom\bigl(\kF, \kG(x) \otimes \CC_y\bigr)$.
It remains to observe that $\ev^{\kF, \, \kG(x)}_y$ is the composition of the following canonical isomorphisms:
$$
\Hom\bigl(\kF, \kG(x)\bigr) \lar \Hom\bigl(\kF, \kG(x) \otimes \CC_y\bigr)
\lar \Hom\bigl(\kF, \kG \otimes \CC_y\bigr) \lar
\Lin\bigl(\kF\big|_y, \, \kG\big|_y\bigr).
$$
\end{proof}

\begin{lemma}\label{L:easyfact}
For $B \in \GL_n(\CC)$ and $v \in \CC$ we set $\kF_v =
\kE\bigl(\exp(2\pi i v) B\bigr) \cong \kE(B) \otimes \kL_v$.
Then for any $v_1, v_2 \in \CC$ and $y \in E$ we have:
$$
\mathrm{dim}_\CC\bigl(\Hom\bigl(\kF_{v_1}, \kF_{v_2}(y)\bigr)\bigr) = n^2.
$$
\end{lemma}

\begin{proof}
The vector bundle $\kF_{v_2}(y)$ is semi--stable of slope one.
Hence, we have:
$$
\Ext^1\bigl(\kF_{v_1}, \kF_{v_2}(y)\bigr) \cong
\Hom\bigl(\kF_{v_2}(y), \kF_{v_1}\bigr)^* = 0.
$$
Thus,   the  statement of Lemma is a consequence of  the Riemann--Roch formula.
\end{proof}

\begin{corollary}\label{C:dim-of-Sol} For any $v, y \in \CC$ the dimension of the complex vector space
\begin{equation}
\Sol = \mathsf{Sol}_{B, \, v,  \, y, \,   \tau}
 :=
\left\{\Phi: \CC \lar \Mat_{n \times n}(\CC)
\left|
\begin{array}{l}
\Phi \mbox{\rm{\quad is holomorphic}} \\
\Phi(z+1) = \Phi(z) \\
\Phi(z+\tau) B = e(z) B \Phi(z)
\end{array}
\right.
  \right\}
\end{equation}
is $n^2$, where $e(z) = e(z, v, y, \tau) =  - \exp\bigl(-2 \pi i (z + v - y + \tau)\bigr)$.
\end{corollary}

\begin{proof}
By Theorem \ref{T:AutomFactors} and Lemma \ref{L:LBdegone}, we have an isomorphism of vector spaces
$\Sol \cong \Hom\bigl(\kF_{v_1}, \kF_{v_2}(y)\bigr)$, where $v = v_1 - v_2$.
Hence, by  Lemma \ref{L:easyfact},
 the dimension
of  $\Sol$ is $n^2$. Taking $y = 0 \in E$, we also recover  the first
part of Proposition \ref{P:preparationary}.
\end{proof}

\begin{theorem}\label{T:keythm}
 Let $B \in \GL_n(\CC)$ and  $\omega = \bar{\theta}'(\frac{1+\tau}{2}) dz \in H^0(\Omega_E)$. Let
 $U \subset \CC$ be a small neighborhood of $0$. Using the projection map
 $\pi: \CC \rightarrow E$, we identify $U$ with a small neighborhood of $\pi(0) \in E$. Then for
 all  $v_1, v_2; y_1, y_2 \in U$
 such that  $y_1 \ne y_2$   the following diagram of vector spaces is commutative:
\begin{equation*}
\begin{array}{c}
\xymatrix{
\Lin\bigl(\kF_{v_1}\big|_{y_1}, \kF_{v_2}\big|_{y_1}\bigr) \ar[d] & &
  \Hom\bigl(\kF_{v_1}, \kF_{v_2}(y_1)\bigr)
  \ar[ll]_-{\res^{\kF_{v_1}, \, \kF_{v_2}}_{y_1}(\omega)}
  \ar[rr]^-{\ev^{\kF_{v_1}, \,  \kF_{v_2}(y_1)}_{y_2}} \ar[d]
  & &
\Lin\bigl(\kF_{v_1}\big|_{y_2}, \kF_{v_2}\big|_{y_2}\bigr) \ar[d] \\
\Mat_{n \times n}(\CC) & & \Sol_{B, \, v, \, y_1, \,  \tau} \ar[ll]_{\res_{y_1}} \ar[rr]^{\ev_{y_2}}
& &  \Mat_{n \times n}(\CC),
}
\end{array}
\end{equation*}
where $v = v_1 - v_2$,
the middle vertical arrow is the isomorphism from Theorem \ref{T:AutomFactors}, whereas
the first and  the last vertical arrows are isomorphisms induced by  trivializations
$\gamma$  from Remark \ref{R:TrivialAutomFact}. The maps
$\res_{y_1}$ and $\ev_{y_2}$ are given by the   formulae:
$$
\res_{y_1}\bigl(\Phi(z)\bigr) = \Phi(y_1)  \quad \mbox{\rm and }
\ev_{y_2}\bigl(\Phi(z)\bigr) = \frac{1}{\bar{\theta}(y_2- y_1 + \frac{\tau+1}{2})} \Phi(y_2),
$$
where $\bar{\theta}(y)$ is the third  Jacobian theta--function.
\end{theorem}

\begin{proof}
The proof of this theorem is literally the same as the one given in
\cite[Section 8.2]{BK4}, see in particular \cite[Corollary 8.10]{BK4}.
\end{proof}

\begin{remark}\label{R:resmap-is-iso}
Let $v_1, v_2 \in \CC$ be such that $v_1 - v_2$ does not belong to the lattice
$\Sigma$. By Corollary \ref{C:nomorphdegzerobundles} we get
the vanishing $\Hom(\kF_{v_1}, \kF_{v_2}) = 0 = \Ext(\kF_{v_1}, \kF_{v_2})$.
Next,  Lemma \ref{L:ResidueisIso} implies that
 the morphism $\res^{\kF_{v_1}, \, \kF_{v_2}}_{y_1}(\omega)$
is an isomorphism. The commutativity of the left square  of the
diagram from  Theorem \ref{T:keythm}
implies   that the linear map $\res_{y_1}$ is an isomorphism, too. Setting $y_1 = 0$, we
obtain a proof of   the second part of Proposition \ref{P:preparationary}.
\end{remark}

\subsection{Triple Massey products and the associative Yang--Baxter equation}
In this subsection,  we give a proof of Theorem \ref{T:main}.

\begin{theorem}\label{T:thmcles} Let $B \in \GL_n(\CC)$, $v_1, v_2 \in \CC$ such that
$v = v_1 - v_2 \notin \Sigma$ and  $y_1, y_2 \in \CC$ such that $y_2 - y_1 \notin \Lambda$.
Consider the linear map $\tilde{r}_{B}(v_1, v_2; y_1, y_2): \Mat_{n \times n}(\CC)
\rightarrow \Mat_{n \times n}(\CC)$ defined via the commutative diagram
\begin{equation}\label{E:def-of-sol}
\begin{array}{c}
\xymatrix
{
\Mat_{n \times n}(\CC)  \ar[rr]^{\tilde{r}_{B}(v_1, v_2; \,  y_1, y_2)}  & & \Mat_{n \times n}(\CC) \\
& \Sol_{B, \, v, \, y_1, \,  \tau} \ar[ul]^{\res_{y_1}}
\ar[ur]_{\ev_{y_2}} & \\
}
\end{array}
\end{equation}
where $\res_{y_1}$ and $\ev_{y_2}$ are as in Theorem \ref{T:keythm}.
Let $r_{B}(v_1, v_2;  y_1,  y_2) \in \Mat_{n \times n}(\CC) \otimes \Mat_{n \times n}(\CC) $ be the tensor
corresponding to the linear map $\tilde{r}_{B}(v_1, v_2;  y_1, y_2)$ via the canonical isomorphism
of vector spaces
$$
\Mat_{n \times n}(\CC) \otimes \Mat_{n \times n}(\CC) \lar \Lin\bigl(\Mat_{n \times n}(\CC), \; \Mat_{n \times n}(\CC) \bigr),
$$
which
sends  a simple tensor $X \otimes Y$ to the linear map $Z \mapsto \Tr(XZ) Y$.
Then the  obtained  function of
 four variables
$$
r: \mathbb{C}^4_{(v_1, v_2; \,  y_1, y_2)}  \lar
\Mat_{n \times n}(\mathbb{C}) \otimes \Mat_{n \times n}(\mathbb{C})
$$
satisfies  the following version of the associative Yang--Baxter equation:
\begin{equation}\label{E:AYBE1}
r_B(v_1, v_2; y_1, y_2)^{12} r_B(v_1, v_3; y_2, y_3)^{23} =
r_B(v_1, v_3; y_1, y_3)^{13} r_B(v_3, v_2; y_1, y_2)^{12} +
\end{equation}
$$
+ r_B(v_2, v_3; y_2, y_3)^{23} r_B(v_1, v_2; y_1, y_3)^{13}.
$$
Moreover,  the tensor-valued function $r$ is unitary, i.e.~it satisfies the condition
\begin{equation}\label{E:AYBEunitary}
r_B(v_1, v_2; y_1, y_2)^{12} = - r_B(v_2, v_1; y_2, y_1)^{21}.
\end{equation}
\end{theorem}

\begin{proof} As above, for $v \in \CC$ we set $\kF_v := \kE\bigl(\exp(2 \pi i v) B\bigr)
\cong \kE(B) \otimes \kL_v$. We split  the proof into  the following logical steps.

\medskip
\noindent
$\bullet$ By Corollary \ref{C:nomorphdegzerobundles},  for any $v_1, v_2 \in \CC$ such that
$v = v_1 - v_2 \notin \Sigma$ we have
\begin{equation*}
\Hom(\kF_{v_1}, \kF_{v_2}) = 0 =
\Ext^1(\kF_{v_1}, \kF_{v_2}).
\end{equation*}
For  simplicity of notation we write $\kF_i$ for $\kF_{v_i}$, $i = 1, 2$.
The following linear
map
$$
m_3 = m^{\kF_1, \kF_2}_{y_1, y_2}:  \Hom(\kF_1, \CC_{y_1}) \otimes \Ext^1(\CC_{y_1}, \kF_2)
\otimes \Hom(\kF_2, \CC_{y_2}) \lar
\Hom(\kF_1, \CC_{y_2}),
$$
called \emph{triple Massey product}, is defined as follows.

Let $a \in \Ext^1(\CC_{y_1}, \kF_2)$,
$g \in \Hom(\kF_1, \CC_{y_1})$, $f \in \Hom(\kF_2, \CC_{y_2})$ and
$
0 \rightarrow  \kF_2 \stackrel{\alpha}\rightarrow \kA \stackrel{\beta}\rightarrow  \CC_{y_1} \rightarrow  0
$
be an extension  representing  the element $a$.
The vanishing of $\Hom(\kF_1, \kF_2)$ and $\Ext^1(\kF_1, \kF_2)$ implies that
we can uniquely lift the morphisms $g$ and $f$ to morphisms
 $\tilde{g}: \kF_1 \rightarrow  \kA$ and $\tilde{f}: \kA \rightarrow  \CC_{y_2}$ such that
 $\beta \tilde{g} = g$ and $\tilde{f} \alpha = f$.  So, we
obtain the following commutative  diagram
$$
\xymatrix{
         &              &           & \kF_1 \ar[d]^g  \ar[dl]_{\tilde g}    &  \\
a: \;  0 \ar[r] & \kF_2 \ar[r]^\alpha \ar[d]_f & \kA \ar[r]^\beta \ar[dl]^{\tilde f}
& \CC_{y_1} \ar[r] &  0 \\
         & \CC_{y_2}    &           &                  & \\
}
$$
and the triple Massey product is defined as  $m_3(g\otimes a \otimes f) =
\tilde{f}\tilde{g}$.

\medskip
\noindent
$\bullet$
By the Serre duality, we have  $\Ext^1(\CC_{y_1}, \kF_2)^* \cong \Hom(\kF_2, \CC_{y_1})$. Let
\begin{equation}\label{E:newFormMassey}
\widetilde{m}^{\kF_1, \kF_2}_{y_1, y_2}:
\Hom(\kF_1, \CC_{y_1}) \otimes \Hom(\kF_2, \CC_{y_2}) \lar
\Hom(\kF_2, \CC_{y_1}) \otimes
\Hom(\kF_1, \CC_{y_2})
\end{equation}
be the image of
 $m^{\kF_1, \kF_2}_{y_1, y_2}$
under the canonical isomorphism of vector spaces
$$
\Lin\bigl(\Hom(\kF_1, \CC_{y_1}) \otimes \Ext^1(\CC_{y_1}, \kF_2)
\otimes \Hom(\kF_2, \CC_{y_2}),
\Hom(\kF_1, \CC_{y_2})\bigr)
\cong
$$
$$
\Lin\bigl(\Hom(\kF_1, \CC_{y_1}) \otimes \Hom(\kF_2, \CC_{y_2}),
\Hom(\kF_2, \CC_{y_1}) \otimes
\Hom(\kF_1, \CC_{y_2})\bigr).
$$
By \cite[Theorem 1]{Polishchuk1} the linear map
$\widetilde{m}^{\kV_1, \kF_2}_{y_1, y_2}$ satisfies
the following ``triangle equation''
\begin{equation}\label{E:yb1}
(\widetilde{m}^{\kF_3, \kF_2}_{y_1, y_2})^{12} (\widetilde{m}^{\kF_1, \kF_3}_{y_1, y_3})^{13} -
(\widetilde{m}^{\kF_1, \kF_3}_{y_2, y_3})^{23} (\widetilde{m}^{\kF_1, \kF_2}_{y_1, y_2})^{12} +
(\widetilde{m}^{\kF_1, \kF_2}_{y_1, y_3})^{13} (\widetilde{m}^{\kF_2, \kF_3}_{y_2, y_3})^{23} = 0.
\end{equation}
Both sides of the equality (\ref{E:yb1}) are viewed  as linear maps
$$
\Hom(\kF_1, \CC_{y_1}) \otimes \Hom(\kF_2, \CC_{y_2})  \otimes
\Hom(\kF_3, \CC_{y_3}) \lar
$$
$$
\lar \Hom(\kF_2, \CC_{y_1}) \otimes \Hom(\kF_3, \CC_{y_2})  \otimes
\Hom(\kF_1, \CC_{y_3}).
$$
Moreover, the tensor $\widetilde{m}^{\kF_1, \kF_2}_{y_1, y_2}$ is
non-degenerate and skew-symmetric:
\begin{equation}\label{E:unitarity}
\rho(\widetilde{m}^{\kF_1, \kF_2}_{y_1, y_2}) = - \widetilde{m}^{\kF_2, \kF_1}_{y_2, y_1},
\end{equation}
where $\rho$ is the isomorphism $$\Hom(\kF_1, \CC_{y_1}) \otimes
\Hom(\kF_2, \CC_{y_2}) \lar
\Hom(\kF_2, \CC_{y_2}) \otimes \Hom(\kF_1, \CC_{y_1})$$
 given by the rule
$\rho(f\otimes g) = g \otimes f$.

\medskip
\noindent
$\bullet$
The   \emph{idea of the proof} of the relation (\ref{E:yb1}) is the following.
 Since the derived category
$D^b\bigl(\Coh(E)\bigr)$ has a structure of an  $A_\infty$--category,
we have the
equality:
\begin{equation}\label{E:AinftyConstraint}
m_3 \circ (m_3 \otimes \mathbbm{1} \otimes \mathbbm{1} +
\mathbbm{1} \otimes m_3 \otimes \mathbbm{1} + \mathbbm{1} \otimes
\mathbbm{1} \otimes m_3) = 0,
\end{equation}
where both sides are  viewed as linear operators mapping the tensor product
$$
\Hom(\kF_1, \CC_{y_1}) \otimes \Ext^1(\CC_{y_1}, \kF_2) \otimes \Hom(\kF_2, \CC_{y_2}) \otimes \Ext^1_E(\CC_{y_2}, \kF_3) \otimes \Hom(\kF_3, \CC_{y_3})
$$
to the vector space $\Hom(\kF_1, \CC_{y_3})$. In other words, the Yang--Baxter relation
(\ref{E:yb1}) is just a translation of the
$A_\infty$--constraint (\ref{E:AinftyConstraint}). Similarly, the unitarity property
(\ref{E:unitarity}) of $\widetilde{m}^{\kF_1, \kF_2}_{y_1, y_2}$
is a consequence of existence of a cyclic $A_\infty$--structure on
$D^b\bigl(\Coh(E)\bigr)$, see \cite[Section 1]{Polishchuk1} for more   details.

\medskip
\noindent
$\bullet$
Consider the linear map
$$\tilde{r}^{\kF_1, \kF_2}_{y_1, y_2}:
\Lin\bigl(\kF_1\big|_{y_1}, \kF_2\big|_{y_1}\bigr)
\lar
\Lin\bigl(\kF_1\big|_{y_2}, \kF_2\big|_{y_2}\bigr)
$$
defined  by the following commutative diagram of vector spaces:
\begin{equation}\label{E:main-diagram}
\begin{array}{c}
\xymatrix@C-C@R+3mm
{
& \Hom\bigl(\kF_1, \kF_2(y_1)\bigr)
\ar[ld]_{\res_{y_1}^{\kF_1, \,  \kF_2}(\omega)}
\ar[rd]^{\ev_{y_2}^{\kF_1, \, \kF_2(y_1)}} &
\\
\Lin\bigl(\kF_1\big|_{y_1}, \kF_2\big|_{y_1}\bigr)
\ar[rr]^{\tilde{r}^{\kF_1, \kF_2}_{y_1, y_2}} & &
\Lin\bigl(\kF_1\big|_{y_2}, \kF_2\big|_{y_2}\bigr).
}
\end{array}
\end{equation}
Then
$
\tilde{r}^{\kF_1, \kF_2}_{y_1, y_2}$ is the image of
 $\widetilde{m}^{\kF_1, \kF_2}_{y_1, y_2}$ under the canonical isomorphism
of vector spaces
$$
\Lin\bigl(\Hom(\kF_1, \CC_{y_1}) \otimes \Hom(\kF_2, \CC_{y_2}),
\Hom(\kF_2, \CC_{y_1}) \otimes
\Hom(\kF_1, \CC_{y_2})\bigr) \cong
$$
$$
\Lin\Bigl(
\Lin\bigl(\kF_1\big|_{y_1}, \kF_2\big|_{y_1}\bigr), \;
\Lin\bigl(\kF_1\big|_{y_2}, \kF_2\big|_{y_2}\bigr)\Bigr),
$$
see  \cite[Theorem 4.17]{BK4} and \cite[Theorem 4]{Polishchuk1}.

\medskip
\noindent
$\bullet$ Let $A$  be an arbitrary automorphy factor, $\kF = \kE(A)$ and $y \in {E}$. Then
we   have an isomorphism of vector spaces
$$
\gamma(A,  y): \Hom(\kF, \CC_y)
\lar \Hom(\kF\otimes \CC_y, \CC_y)
\lar \kF\big|^{*}_{y}  \lar \CC^n,
$$
induced by  the  trivialization $\gamma_A$ from Remark \ref{R:TrivialAutomFact}.
For any $v \in \CC$ we denote  by $\gamma(v, y)$ the isomorphism
$\Hom(\kF_v, \CC_y) \rightarrow  \CC^n$.
We obtain a linear map $\bar{r}_B(v_1, v_2; y_1, y_2)$,
defined by the following commutative diagram of vector spaces:
$$
\xymatrix
{
\Hom(\kF_{v_1}, \CC_{y_1}) \otimes \Hom(\kF_{v_2}, \CC_{y_2})
\ar[rr]^{\widetilde{m}^{\kF_1, \kF_2}_{y_1, y_2}}
\ar[d]_{\gamma(v_1, \,  y_1) \, \otimes \,  \gamma(v_2, \,  y_2)} & &
\Hom(\kF_{v_2}, \CC_{y_1}) \otimes \Hom(\kF_{v_1}, \CC_{y_2})
\ar[d]^{\gamma(v_2, \,  y_1) \, \otimes  \, \gamma(v_1, \,  y_2)} \\
\CC^{n} \otimes \CC^n \ar[rr]^{\bar{r}_B(v_1, v_2;\,  y_1, y_2)} & & \CC^{n} \otimes \CC^n.
}
$$

\medskip
\noindent
$\bullet$ Using the canonical isomorphism
$\Lin(\CC^n \otimes \CC^n, \,\CC^n \otimes \CC^n) \rightarrow \Mat_{n \times n}(\CC)
\otimes \Mat_{n \times n}(\CC)$, we end up with a tensor-valued meromorphic function
$$
\CC^2_{(v_1, v_2)} \times \CC^2_{(y_1, y_2)}
 \stackrel{r_B}\lar \Mat_{n \times n}(\CC) \otimes \Mat_{n \times n}(\CC),
$$
satisfying the Yang--Baxter equation (\ref{E:AYBE1}) and the unitarity condition
(\ref{E:AYBEunitary}). Moreover, for any $v_1, v_2; y_1, y_2 \in \CC$ such that
$v_1 - v_2 \notin \Sigma$ and $y_1 - y_2 \notin \Lambda$ the tensor
$r_B(v_1, v_2; y_1, y_2)$ coincides with the image of $
\tilde{r}^{\kF_1, \kF_2}_{y_1, y_2}$ under the composition of the canonical
isomorphism of vector spaces
$$
\Lin\Bigl(
\Lin\bigl(\kF_1\big|_{y_1}, \kF_2\big|_{y_1}\bigr), \;
\Lin\bigl(\kF_1\big|_{y_2}, \kF_2\big|_{y_2}\bigr)\Bigr)
\lar \Lin\bigl(\kF_2\big|_{y_1}, \kF_1\big|_{y_1}\bigr) \otimes
\Lin\bigl(\kF_1\big|_{y_2}, \kF_2\big|_{y_2}\bigr)\Bigr)
$$
with the isomorphism
$ \Lin\bigl(\kF_1\big|_{y_2}, \kF_2\big|_{y_2}\bigr)\Bigr) \lar
\Mat_{n \times n}(\CC) \otimes \Mat_{n \times n}(\CC)
$
induced by trivializations
$\gamma$ from Remark \ref{R:TrivialAutomFact}.
\end{proof}

\begin{remark}
 In order to prove that the function
$r = r_B(v_1, v_2; y_1, y_2)$ is actually holomorphic  with
an analytic  dependence on the entries of the matrix $B$,   we need again
the formalism of sheaves.  This will be done in Subsection \ref{SS:sheaves}. The reader, interested
in the actual solutions may go directly  to Section \ref{S:Kalkul}.
\end{remark}

\subsection{Remarks on the constructed solutions} In the previous subsection
we have seen how one can attach to a matrix $B \in \GL_n(\CC)$ a unitary solution
$$
\CC^2 \times \CC^2 \stackrel{r_B}\lar \Mat_{n \times n}(\CC) \otimes
\Mat_{n \times n}(\CC)
$$
of the associative  Yang--Baxter equation (\ref{E:AYBE1}), see  diagram (\ref{E:def-of-sol})
from Theorem \ref{T:thmcles}.

\begin{proposition}\label{P:transl-invariance}
For  general $v_1, v_2, u; y_1, y_2, x \in \CC$ we have the equality
$$r_B(v_1 + u, v_2 + u; \,  y_1 +x, y_2 + x) = r_B(v_1, v_2; \, y_1, y_2).
$$
In other words, the function $r_B(v_1, v_2; y_1, y_2)$ depends only on the differences
$v = v_1 - v_2$ and $y = y_2 - y_1$. In particular,
 the function $r_B(v,  y) = r_B(v_1, v_2; y_1, y_2)$ satisfies the
associative Yang--Baxter equation (\ref{E:AYBE}).
\end{proposition}

\begin{proof}
Since the vector space $\Sol_{B, \, v, \, y_1,\, \tau}$ from Theorem \ref{T:keythm} only
depends on the difference $v = v_2 - v_1$, whereas $\res_{y_1}$ and $\ev_{y_2}$ depend only on $y_1$ and
 $y_2$, we have the equality
$r_B(v_1 + u, v_2 + u; y_1, y_2) = r_B(v_1, v_2; y_1, y_2)$.
To show the translation invariance of the function  $r_B$
with respect to the second pair of spectral variables note that we have the following
commutative diagram:
$$
\xymatrix
{
 & \Sol_{B, \, v,\, y_1, \, \tau} \ar[ld]_{\res_{y_1}} \ar[rd]^{\ev_{y_2}} \ar[dd]^{t_x} & \\
 \Mat_{n\times n}(\CC) & &  \Mat_{n\times n}(\CC) \\
 & \Sol_{B, \, v,\, y_1+x, \, \tau} \ar[lu]^{\res_{y_1+x}} \ar[ru]_{\ev_{y_2+x}}  &
}
$$
where $t_x\bigl(\Phi(z)\bigr) = \Phi(z-x)$. It proves that  $r_B(v_1, v_2; y_1 + x, y_2 + x) = r_B(v_1, v_2; y_1, y_2).$
\end{proof}

\begin{remark}
Proposition \ref{P:transl-invariance} implies that in order to compute the linear map $r_B(v, y)$
we can take
$y_1 = 0$ and $y_2 = y$
 in the commutative diagram (\ref{E:def-of-sol}). In particular,
the solution $r_B(v, y)$ can be computed using the  diagram (\ref{E:keydiagram}).
\end{remark}

\begin{proposition}\label{P:gauge-transf}
Let $B, S \in \GL_n(\CC)$ and $A := S^{-1} B S$. Then we have:
$$
r_A(v, y) = \bigl(S^{-1} \otimes S^{-1} \bigr) \; r_B(v, y)  \; \bigl(S \otimes S\bigr).
$$
\end{proposition}

\begin{proof} For  simplicity of notation  we  denote $\Sol_B = \Sol_{B, \, v,\, y,\, \tau}$
and $r_B = r_B(v, y)$.
Observe  that we have an isomorphism of vector spaces
$\varphi_S: \Sol_{B} \rightarrow \Sol_{A}$ mapping a function $\Phi \in \Sol_B$ to
$S^{-1} \, \Phi \, S \in \Sol_A$.  We have a commutative diagram
\begin{equation}\label{E:diag-for-gauge}
\begin{array}{c}
\xymatrix{
\Mat_{n\times n}(\CC) \ar[d]_{c_S} & \Sol_{B} \ar[l]_-{\res_0} \ar[r]^-{\ev_y} \ar[d]^{\varphi_S} & \Mat_{n\times n}(\CC) \ar[d]^{c_S}\\
\Mat_{n\times n}(\CC)        & \Sol_{A} \ar[l]_-{\res_0} \ar[r]^-{\ev_y}         & \Mat_{n\times n}(\CC),
}
\end{array}
\end{equation}
where $c_S(X) = S^{-1} X S$. This
implies that for any $X \in \Mat_{n \times n}(C)$ we have:
$
\tilde{r}_{A}(S^{-1} X S) = S^{-1} \tilde{r}_B(X)  S.
$
The matrix $S$ defines the following linear automorphism
$$
\psi_S: \Lin\bigl(\Mat_{n \times n}(\CC), \; \Mat_{n \times n}(\CC) \bigr)
\lar
\Lin\bigl(\Mat_{n \times n}(\CC), \; \Mat_{n \times n}(\CC) \bigr)
$$
sending $l \in \Lin\bigl(\Mat_{n \times n}(\CC), \, \Mat_{n \times n}(\CC)\bigr)$ to the linear map
$X \xrightarrow{\psi_S(l)} S^{-1} \, l(S X S^{-1}) \, S$. Then we have:
$\psi_S(\tilde{r}_B) = \tilde{r}_A$.

Finally, let  $\mathsf{can}: \Mat_{n \times n}(\CC) \otimes \Mat_{n \times n}(\CC) \rightarrow
\Lin\bigl(\Mat_{n \times n}(\CC), \, \Mat_{n \times n}(\CC)\bigr)$ be the canonical
isomorphism of vector spaces mapping a simple tensor $X \otimes Y$ to the linear
map $Z \mapsto \Tr(XZ) Y$. Then the following diagram is commutative:
$$
\xymatrix
{
\Mat_{n \times n}(\CC) \otimes \Mat_{n \times n}(\CC) \ar[d]_{c_S \, \otimes \,  c_S}
\ar[rr]^-{\mathsf{can}} & & \Lin\bigl(\Mat_{n \times n}(\CC), \, \Mat_{n \times n}(\CC)\bigr)
\ar[d]^{\psi_S}\\
\Mat_{n \times n}(\CC) \otimes \Mat_{n \times n}(\CC) \ar[rr]^-{\mathsf{can}} & & \Lin\bigl(\Mat_{n \times n}(\CC), \, \Mat_{n \times n}(\CC)\bigr).
}
$$
But this implies that  $r_A(v, y) = \bigl(S^{-1} \otimes S^{-1} \bigr) \; r_B(v, y)  \; \bigl(S \otimes S\bigr)$.
\end{proof}

\subsection{Semi--universal family of degree zero semi-stable vector bundles on a complex torus
and the associative
Yang--Baxter equation}\label{SS:sheaves}
 In the previous subsections we have explained how
one can attach to a matrix $B \in \GL_n(\CC)$ a unitary solution $r_B(v, y)$ of the associative
Yang--Baxter equation (\ref{E:AYBE}). However, it still remains to be shown that
$r_B$ is a meromorphic  function in $v$ and $y$, holomorphic on
$(\CC \setminus \Sigma) \times (\CC \setminus \Lambda)$ and with an analytic dependence
on the matrix $B$. Although this fact can be verified   by a direct computation, we
prefer to give an abstract proof based on the technique of semi--universal families of semi--stable sheaves.

\medskip
\noindent
$\bullet$ Let $G = \GL_n(\CC)$ and $\kP \in \VB(E \times G)$ be defined as follows
$$\kP: =  \CC \times G \times \CC^n/\sim, \quad \mbox{\rm where} \quad
(z, g, v) \sim (z + 1, g, v) \sim (z + \tau, g, g \cdot v)$$
for all $(z, g, v) \in \CC \times G \times \CC^n$.
 Note that we have
a Cartesian diagram
$$
\xymatrix
{
(\CC \times G) \times \CC^n \ar[r] \ar[d]_{\mathrm{pr}_1} & \kP
\ar[d]\\
\CC \times G \ar[r]^{\pi \times \mathbbm{1}}  & E \times G,
}
$$
where $\pi: \CC \rightarrow \CC/\Lambda = E$ is the quotient  map.
Note that for any $g \in G$ we have an isomorphism
$\kP\big|_{E \times \{g\}} \cong \kE(g)$, where $\kE(g)$ is the semi--stable degree zero
vector bundle on $E$ determined  by the automorphy factor $g \in \GL_n(\CC)$. Thus, the
 constructed vector bundle $\kP$ is a \emph{semi--universal} family of degree zero
 semi-stable vector bundle on the torus $E$.

 \medskip
\noindent
$\bullet$ Let $J = \Pic^0(E)$ be the Jacobian of $E$. One can identify
$J$ with the torus $E$ using the following construction. Consider
the line bundle $\kL$ on $E \times E = \CC/\Lambda \times \CC/\Lambda$ defined as the quotient space
$\kL: =  (\CC \times \CC)  \times \CC/\sim$, where
$$
(z, w, v) \sim (z + 1, w, v) \sim (z, w+1, v) \sim (z, w + \tau , v) \sim
(z + \tau, w, \exp(2 \pi i w)v)
$$
for all $(z, w, v) \in (\CC \times \CC) \times \CC$. The constructed line bundle
$\kL$ is a universal family of degree zero vector bundles on $E$.

 \medskip
\noindent
$\bullet$
We denote  $X = E \times J \times J \times E \times E \times G$, $T = J \times J \times E \times E \times G$ and set
 $q: X \rightarrow T$ and $p: X \rightarrow E \times G$
to be the canonical projection maps.
Similarly, for $i = 1, 2$ we define
 $p_i: X \rightarrow E \times J$ and
 $h_i: T \rightarrow X$ to be given by the formulae
  $p_i(x, v_1, v_2, y_1, y_2, g) = (x, v_i)$ and
  $h_i(v_1, v_2, y_1, y_2, g) = (y_i, v_1, v_2, y_1, y_2, g)$. Note that
  $h_1$ and $h_2$ are sections of the canonical projection $q$.

\medskip
\noindent
$\bullet$ For $i = 1, 2$ we define $\kF_i := p^*\kP \otimes p_i^{*} \kL$.
Obviously, for any point $t= (v_1, v_2, y_1, y_2, g) \in T$ we have:
$
\kF_i\big|_{q^{-1}(t)} \cong \kP\big|_{E \times \{g\}} \otimes \kL\big|_{E \times \{v_i\}} \cong
\kE\bigl(\exp(2 \pi v_i) \cdot g\bigr).
$
\begin{lemma}\label{L:corGrauert}
The coherent sheaf $q_* {\mathcal Hom}_X(\kF_1, \kF_2)$ is supported on a proper
closed analytic subset of $T$.
\end{lemma}

\begin{proof} By Grauert's direct image theorem, the sheaf
  $q_* {\mathcal Hom}_X(\kF_1, \kF_2)$ is coherent, hence it is supported on a closed analytic subset
  $\Delta$ of the base  $T$.
Since ${\mathcal Hom}_X(\kF_1, \kF_2)$ is a vector bundle on $X$, it is flat over $T$
and for any point $t  = (v_1,\,  v_2,\,  y_1,\,  y_2,\,  g) \in T$ we have a base--change isomorphism
$$
q_* {\mathcal Hom}_X(\kF_1, \kF_2) \otimes \CC_{t} \cong
\Hom_E\bigl(
\kE(g) \otimes \kL_{v_1}, \kE(g) \otimes \kL_{v_2}\bigr).
$$
By Corollary \ref{C:nomorphdegzerobundles}, we have the vanishing $\Hom_E\bigl(
\kE(g) \otimes \kL_{v_1}, \kE(g) \otimes \kL_{v_2}\bigr) = 0$  for
generically chosen  $v_1, v_2 \in J$ and $g \in G$. Hence,   $\Delta$
is a proper subset of $T$.
\end{proof}

\noindent
The following result  can be proven along the same lines as Lemma \ref{L:corGrauert}.

\begin{lemma}
Let $D_i := \mathsf{Im}(h_i) \subseteq {X}$. Then the sheaf
$q_* {\mathcal Hom}_X\bigl(\kF_1(D_2), \kF_2(D_1)\bigr)$ is supported on a proper closed analytic
subset $\Delta'$ of $T$ and $q_* {\mathcal Hom}_{{X}}\bigl(\kF_1, \kF_2(D_1)\bigr)$ is
a vector vector bundle of rank $n^2$.
\end{lemma}

\medskip
\noindent
$\bullet$  Let $\breve{T} := T \, \setminus \, (\Delta \cup \Delta')$ and  $\breve{X} := q^{-1}(\breve{T})$.
For the  sake of simplicity we
 denote the restrictions of $\kF_1$ and $\kF_2$ on $\breve{X}$ by the same symbols.
 Let $\omega \in H^0\bigl(\Omega_{\breve{X}/\breve{T}}\bigr)$ be the pull-back of the
 differential form $dz \in H^0\bigl(\Omega_E)$. Note that we are in the situation of \cite[Section 5.3]{BK4}. In particular, we have the following  commutative diagram in $\Coh(\breve{T})$, where all arrows are isomorphisms
  of vector bundles on $\breve{T}$:
 \begin{equation}\label{E:diag-of-sheaves}
 \begin{array}{c}
 \xymatrix{
  & q_* {\mathcal Hom}_{\breve{X}}\bigl(\kF_1, \kF_2(D_1)\bigr) \ar[ld]_{\res_{h_1}^{\kF_1, \, \kF_2}(\omega)}
  \ar[rd]^{\ev^{\kF_1, \, \kF_2(D_1)}_{h_2}} & \\
  {\mathcal Hom}_{\breve{T}}\bigl(h_1^*\kF_1, h_1^*\kF_2\bigr)
  \ar[rr]^{\tilde{r}^{\kF_1, \, \kF_2}_{h_1, \, h_2}} & & {\mathcal Hom}_{\breve{T}}\bigl(h_2^*\kF_1, h_2^*\kF_2\bigr).
 }
 \end{array}
 \end{equation}
 The morphisms $\res_{h_1}^{\kF_1, \, \kF_2}(\omega)$ and $\ev^{\kF_1, \, \kF_2(D_1)}_{h_2}$
 are induced by the short exact sequences
 $$
 0 \rightarrow \Omega_{\breve{X}/\breve{T}} \rightarrow
 \Omega_{\breve{X}/\breve{T}}(D_1) \stackrel{\res_{D_1}}\lar \kO_{D_1} \rightarrow 0, \quad
 0 \rightarrow \kO_{\breve{X}}(-D_2) \rightarrow \kO_{\breve{X}} \rightarrow \kO_{D_1} \lar 0,
 $$
 see \cite[Section 5.3]{BK4}.
  By \cite[Theorem 5.17]{BK4}, after tensoring the diagram (\ref{E:diag-of-sheaves}) with $\CC_{t}$, where $t = (v_1, v_2, y_1, y_2, g) \in \breve{T}$,
  and
 applying  base change isomorphisms, we get  the commutative diagram   (\ref{E:main-diagram}).
 In particular, the function $\tilde{r}_{B}(v_1, v_2; y_1, y_2)$
 from Theorem \ref{T:thmcles} is just the  isomorphism of vector bundles
 $\tilde{r}^{\kF_1, \, \kF_2}_{h_1, \, h_2}$ written with respect of the trivialization
 $\gamma$, described in  Remark \ref{R:TrivialAutomFact}. This implies that the tensor $r_B(v, y)$ is
 non-degenerate.

 In a similar way, the isomorphism $\tilde{r}_{B}(v_1, v_2; y_1, y_2)$ determines
  a holomorphic section
 ${r}^{\kF_1, \, \kF_2}_{h_1, \, h_2} \in H^0\bigl(\breve{T},
 {\mathcal Hom}_{\breve{T}}\bigl(h_1^*\kF_2, h_1^*\kF_1\bigr) \otimes
 {\mathcal Hom}_{\breve{T}}\bigl(h_2^*\kF_1, h_2^*\kF_2\bigr) \bigr)$.
 Trivializing $\kF_1$ and $\kF_2$ as in Remark \ref{R:TrivialAutomFact}, the section
 ${r}^{\kF_1, \, \kF_2}_{h_1, \, h_2}$ becomes  the tensor-valued function $r_B(v, y)$
 from Theorem \ref{T:thmcles}.
 This proves that
 ${r}_{B}(v, y)$ is holomorphic on $(\CC \setminus \Sigma_B) \times (\CC \setminus  \Lambda)$ and
   as a function  of the input matrix $B$. To show that ${r}_{B}(v, y)$ is meromorphic on
   $\CC \times \CC$ note that $\res_{h_1}^{\kF_1, \, \kF_2}(\omega)$ and $\ev^{\kF_1, \, \kF_2(D_1)}_{h_2}$
   are morphisms of vector bundles of rank $n^2$ on the whole base $T$ and
    ${\tilde{r}}^{\kF_1, \, \kF_2}_{h_1, \, h_2} = \ev^{\kF_1, \, \kF_2(D_1)}_{h_2} \circ
   \bigl(\res_{h_1}^{\kF_1, \, \kF_2}(\omega)\bigr)^{-1}$ is a  meromorphic isomorphism of
   ${\mathcal Hom}_{{T}}\bigl(h_1^*\kF_1, h_1^*\kF_2\bigr)$ and
   ${\mathcal Hom}_{{T}}\bigl(h_2^*\kF_1, h_2^*\kF_2\bigr)$.

 \section{Computations of solutions of AYBE}\label{S:Kalkul}
 
 In this section we compute the solutions of the associative Yang--Baxter equation 
 (\ref{E:AYBE}) attached to a diagonal matrix and to a Jordan block.

 \subsection{Solution obtained from a diagonal matrix}
 All or  computations   are based on the following standard fact.

 \begin{lemma}\label{L:basic-for-theta} Let $\varphi(z) = \exp(-\pi i \tau - 2 \pi i z)$. Then the
 vector space
 \begin{equation}
\left\{f: \CC \lar \CC
\left|
\begin{array}{l}
f \mbox{\rm{\quad is holomorphic}} \\
f(z+1) = f(z) \\
f(z+\tau)  = \varphi(z) f(z)
\end{array}
\right.
  \right\}
\end{equation}
is one--dimensional and generated by the third Jacobian theta--function
$$
\bar{\theta}(z) = \theta_3(z|\tau) = \sum\limits_{n \in \mathbb{Z}}
\exp(\pi i n^2 \tau + 2 \pi i n z).
$$
 \end{lemma}

 \begin{proof}
 A proof of this result can  for instance be found in \cite[Chapter 1]{MumfordTheta}.
 \end{proof}

 \begin{theorem}
 Let  $B = \mathsf{diag}\bigl(\exp(2\pi i \lambda_1), \dots, \exp(2\pi i \lambda_n)\bigr)$
for some $\lambda_1, \dots, \lambda_n \in \CC$.
Then the corresponding solution of the associative Yang--Baxter equation
described in Theorem \ref{T:thmcles}  is given by the  following formula:
\begin{equation}\label{E:sol-for-diag}
r_B(v, y) =
\sum\limits_{k, \, l = 1}^n \sigma(v - \lambda_{kl}, y) \, e_{l, k} \otimes e_{k, l},
\end{equation}
where $\lambda_{kl} = \lambda_k - \lambda_l$ for all $1 \le k, l \le n$ and $\sigma(u, x)$ is the Kronecker function.
 \end{theorem}

 \begin{proof}
 Let $\Phi(z) = \bigl(a_{kl}(z)\bigr)$ be an element of $\Sol = \Sol_{B, \, v,\, 0, \,\tau}$, where
 $v = v_1 - v_2$.
 Then for all $1 \le k, l \le n$ we have:
 $$
 \left\{
 \begin{array}{ccl}
 a_{kl}(z+1) & = &  a_{kl}(z) \\
 a_{kl}(z+\tau) & = & \exp\bigl(- \pi i \tau - 2 \pi i (z + v + {\frac{\tau +1}{2}}
   - \lambda_{kl})\bigr) a_{kl}(z).
 \end{array}
 \right.
 $$
 Hence, there exist $\beta_{kl} \in \CC$ such that
 $a_{kl}(z) = \beta_{kl} \bar{\theta}(z + v + \frac{\tau +1}{2} - \lambda_{kl})$.

 If
 $A = (\alpha_{kl}) \in \Mat_{n \times n}(\CC)$ is such that
 $\res_0\bigl(\Phi(z)\bigr) = A$ then
 $\beta_{kl} = \frac{1}{\bar{\theta}(v + \frac{\tau +1}{2} - \lambda_{kl})} \alpha_{kl}$.
 If $C = (\gamma_{kl}) := \ev_{y}\bigl(\Phi(z)\bigr)$ then for all $1 \le k, l \le n$ we have
 $$
 \gamma_{kl} = \frac{\bar{\theta}(v + y - \lambda_{kl} + \frac{\tau +1}{2})}{
 \bar{\theta}(v - \lambda_{kl} + \frac{\tau +1}{2}) \, \bar{\theta}(y + \frac{\tau +1}{2})} \alpha_{kl}
 =
 \frac{1}{i \exp(-\pi i \frac{\tau}{4})} \frac{\theta(v-\lambda_{kl} + y)}{\theta(v - \lambda_{kl})
 \theta(y)} \alpha_{kl},
 $$
 where we have used the well--known relation between the first and the third Jacobian theta functions
 $
 \bar{\theta}(z + \frac{\tau +1}{2}) = i \exp(-\pi i (z + \frac{\tau}{4})) \theta(z).
 $
 Hence, the linear map $\tilde{r}_B(v, y): \Mat_{n \times n}(\CC) \rightarrow
 \Mat_{n \times n}(\CC)$ sends  the basis vector $e_{k, l}$ to
 $\frac{\exp(\pi i \frac{\tau}{4})}{i \theta'(0)} \, \sigma(v-\lambda_{kl}, y) e_{k,  l}$. Neglecting
 the constant $\frac{\exp(\pi i \frac{\tau}{4})}{i \theta'(0)}$, we end  up
 with the solution $r_B(v, y)$ given by   (\ref{E:sol-for-diag}).
 \end{proof}
 
 \subsection{Solution attached to a Jordan block} In this subsection we compute the
 solution of the associative Yang--Baxter equation (\ref{E:AYBE}) attached to a Jordan 
 block of size $n \times n$. Fist note the following easy fact. 

\begin{lemma}
For any $n \in \mathbb{N}$ and $\lambda \in \CC^*$ the solution $r_{J_n(\lambda)}(v, y)$
constructed in Theorem \ref{T:thmcles}, is gauge
equivalent to $r_J(v, y)$, where $J_n(\lambda)$ is the Jordan block of size $n\times n$ with
eigenvalue $\lambda$ and $J = J_n(1)$.
\end{lemma}

\begin{proof}
Since the matrices $J_n(\lambda)$ and $\lambda \cdot J$ are conjugate,  Proposition
\ref{P:gauge-transf} implies that the corresponding solutions are gauge equivalent. From the
algorithm of the construction of solutions of (\ref{E:AYBE}) presented in Theorem \ref{T:thmcles}
it is clear that the matrices $\lambda \cdot J$ and $J$ give the same solutions.
\end{proof}

\medskip
\noindent
Hence, it suffices  to describe the solution of the associative Yang--Baxter equation (\ref{E:AYBE}) attached
to the Jordan block $J$.

\begin{definition}\label{D:defin-of-matrixN} Let $n \in \mathbb{N}$ be fixed. For all $1 \le k \le n-1$
we set $a_{k}=\frac{(-1)^{k}}{k}$,
\[
A_{0}=\left(
\begin{array}{cccc}
0 & \cdots & \cdots & 0\\
a_{1} & \ddots &  & \vdots\\
\vdots & \ddots & \ddots & \vdots\\
a_{n-1} & \cdots & a_{1} & 0
\end{array}
\right) \quad \mbox{and} \quad  A_{k}=-a_{k}\,\cdot\mathbbm{1}_{n\times n}.
\]
Next, consider the following matrix $N$ from $\Mat_{n^2 \times n^2}(\CC)$:
\begin{equation}
N=\left(\begin{array}{cccc}
A_{0} & A_{1} & \cdots & A_{n-1}\\
0 & \ddots & \ddots & \vdots\\
\vdots & \ddots & \ddots & A_{1}\\
0 & \cdots & 0 & A_{0}\end{array}\right).
\end{equation}
\end{definition}

\noindent
Note the following easy fact.

\begin{lemma}
The matrix $N$ is nilpotent. More precisely, $N^{2n-1}=0$.
\end{lemma}

\begin{definition}\label{D:def-of-nabla} Consider the differential operator
 $\nabla=-\frac{1}{2\pi i}\cdot\frac{d}{dz}$
acting on the vector space $\mathcal{M}$ of meromorphic functions on $\CC$.
For  all  $0\leq k,  l\leq n-1$ we define the  linear
operator $\nabla_{k, l}: \kM \rightarrow \kM$ given by the following formula:
\begin{equation}\label{E:def-of-nabla}
\nabla_{k, l}=e_{n(n-k-1)+l+1}^{t}\,\exp\left(\nabla\, N\right)\, e_{n(n-1)+1}.
\end{equation}
Since the matrix $N$ is nilpotent, the operators $\nabla_{k, l}$ are polynomials in $\nabla$.
Note that $\nabla_{0, 0} = e_{n(n-1)+1}^{t}\,\exp\left(\nabla\, N\right)\, e_{n(n-1)+1}$ is the identity operator.
\end{definition}

\noindent
Now we can state the  main result   of this subsection.

\begin{theorem}\label{thm: semi-simple main formula} Let $J$ be the  Jordan block
of size $n \times n$ with eigenvalue one. Then the corresponding solution
of the associative Yang-Baxter equation, described in Theorem \ref{T:thmcles},
is given by the following formula:
\begin{equation}\label{E:real-beauty}
r_{J}(v,y)=\left(\sum_{\begin{smallmatrix}0\leq k\leq n-1\\
0\leq l\leq n-1\end{smallmatrix}}\nabla_{k,l}\bigl(\sigma\left(v,y\right)\bigr)\sum_{\begin{smallmatrix}1\leq i\leq n-l\\
1\leq j\leq n-k\end{smallmatrix}}e_{i,j+k}\otimes e_{j,i+l}\right),
\end{equation}
where $\sigma(v,y)$ is the Kronecker function and $\nabla_{k, l}$ acts on the first spectral variable.
\end{theorem}

\begin{remark}
Let $1 \le a, b, c, d \le n$. Then the coefficient of the tensor $e_{a, b} \otimes e_{c, d}$
in the expression for $r_J(v, y)$ from Equation (\ref{E:real-beauty}) is zero unless
$d \ge a$ and $b \ge c$. Moreover, this coefficient depends only on the differences
$d-a$ and $b -c$.
\end{remark}

\begin{example}\label{exa: M(2,0) ass. r-Matrix} Let $n=2$ and $J = \left(\begin{smallmatrix} 1 & 1 \\
0 & 1 \end{smallmatrix}\right)$.  Note that
\[  N = \left(
{\begin{tabular}{r c | r c}
0 & 0 & 1 & 0\\
-1 & 0 & 0 & 1\\ \hline
0 & 0 & 0 & 0\\
0 & 0 & -1 & 0
\end{tabular}}
\right)
, \quad
N^2= \left(
{\begin{tabular}{c c | r c}
0 & 0 & 0 & 0\\
0 & 0 & -2 & 0\\ \hline
0 & 0 & 0 & 0\\
0 & 0 & 0 & 0
\end{tabular}}
\right)
\] and that all higher powers of $N$ are zero. Hence,
\[
\mbox{exp}(\nabla\, N)=1+\nabla\, N+\frac{\nabla^{2}\, N^{2}}{2}
=
\left(
\begin{array}{rc|cc}
\mathbbm{1}  & 0 & \nabla & 0 \\
-\nabla & \mathbbm{1} & - \nabla^2 & \nabla \\
\hline
0 & 0 & \mathbbm{1} & 0 \\
0 & 0 & - \nabla & \mathbbm{1}
\end{array}
\right)
\]
 and we derive that
\begin{align*}
r_J(v, y) \; = \;   &
\sigma(v, y) \bigl(e_{11} \otimes e_{11} + e_{22} \otimes e_{22} +
e_{12} \otimes e_{21} + e_{21} \otimes e_{12}\bigr) +  \\
  & \nabla \sigma(v, y)
\bigl(e_{12} \otimes h  - h \otimes e_{12}\bigr) -
 \nabla^2 \sigma (v, y) e_{12} \otimes e_{12},
\end{align*}
where $h = e_{11} - e_{22}$.
\end{example}

\begin{remark}
From the fact that the function $r_J(v, y)$ from Example \ref{exa: M(2,0) ass. r-Matrix}
satisfies the associative Yang--Baxter equation (\ref{E:AYBE}) we obtain the
following identity
for  derivatives of the Kronecker function with respect to the first spectral variable:
\begin{align*}
\sigma'(u, x+y) \sigma'(v, y) - \sigma'(u, x) \sigma'(u+v, y) -
\sigma'(-v, x) \sigma'(u+v, x+y)  \\ = \sigma(u, x) \sigma''(u+v, y) -
\sigma(-v, x) \sigma''(u+v, x+y).
\end{align*}
\end{remark}


\begin{example}
For $n=3$ and $J = \left(\begin{smallmatrix} 1 & 1 & 0 \\
0 & 1  & 1 \\ 0 & 0 & 1\end{smallmatrix}\right)$ we have \[  N = \left(
{\begin{tabular}{r c c| r c c| r r r}
0 & 0 & 0 & 1 & 0 & 0 & -$\frac{1}{2}$ & 0 & 0\\
-1 & 0 & 0 & 0 & 1 & 0 & 0 & -$\frac{1}{2}$ & 0\\
$\frac{1}{2}$ & -1 & 0 & 0 & 0 & 1 & 0 & 0 & -$\frac{1}{2}$\\ \hline
0 & 0 & 0 & 0 & 0 & 0 & 1 & 0 & 0\\
0 & 0 & 0 & -1 & 0 & 0 & 0 & 1 & 0\\
0 & 0 & 0 & $\frac{1}{2}$ & -1 & 0 & 0 & 0 & 1\\ \hline
0 & 0 & 0 & 0 & 0 & 0 & 0 & 0 & 0\\
0 & 0 & 0 & 0 & 0 & 0 & -1 & 0 & 0\\
0 & 0 & 0 & 0 & 0 & 0 & $\frac{1}{2}$ & -1 & 0
\end{tabular}}
\right).
\]
Note that \[
\mbox{exp}(\nabla\, N)\, e_{7}=\left(\begin{array}{c}
0\\
0\\
0\\
0\\
0\\
0\\
1\\
0\\
0\end{array}\right)+
\left(\begin{array}{r}
-\frac{1}{2}\\
0\\
0\\
1\\
0\\
0\\
0\\
-1\\
\frac{1}{2}\end{array}\right) \nabla +
\left(\begin{array}{r}
1\\
1\\
-\frac{1}{2}\\
0\\
-2\\
1\\
0\\
0\\
1\end{array}\right)\frac{\nabla^{2}}{2}+
\left(\begin{array}{r}
0\\
-3\\
0\\
0\\
0\\
3\\
0\\
0\\
0\end{array}\right)\frac{\nabla^{3}}{6}+
\left(\begin{array}{r}
0\\
0\\
6\\
0\\
0\\
0\\
0\\
0\\
0\end{array}\right)\frac{\nabla^{4}}{24}\]
and that $\nabla_{k, l}=e_{3(2-k)+l+1}^{t}\bigl(\mbox{exp}(\nabla\, N)\, e_{7}\bigr)$. Carrying out computations,
we end up with the following solution of the associative Yang--Baxter equation:
\begin{align*}
r_J(v, y) = &   \sigma \sum_{1\leq i, j \leq 3} e_{i,j}\otimes e_{j,i}
+ \nabla\sigma \sum_{\begin{smallmatrix} 1\leq i\leq3\\
1\leq j\leq2\end{smallmatrix}} \bigl(e_{i,j+1}\otimes e_{j,i} - e_{j,i} \otimes e_{i,j+1}\bigr) + \\
& \Bigl(-\frac{1}{2}\nabla+\frac{1}{2}\nabla^{2}\Bigr)\sigma
\sum_{1\leq i\leq 3} e_{i,3}\otimes e_{1,i} + \Bigl(\frac{1}{2}\nabla+\frac{1}{2}\nabla^{2}\Bigr)\sigma\sum_{1\leq i\leq 3} e_{1,i} \otimes e_{i,3} \\
& \Bigl(\frac{1}{2}\nabla^{2}-\frac{1}{2}\nabla^{3}\Bigr)\sigma
\sum_{1\leq i\leq 2}e_{i,3}\otimes e_{1,i+1} + \Bigl(\frac{1}{2}\nabla^{2}+\frac{1}{2}\nabla^{3}\Bigr)\sigma
\sum_{1\leq i\leq 2}  e_{1,i+1} \otimes e_{i,3} + \\
&
-\nabla^{2}\sigma   \sum_{1\leq i, j \leq 2} e_{i,j+1}\otimes e_{j,i+1} + \Bigl(-\frac{1}{4}\nabla^{2}+\frac{1}{4}\nabla^{4}\Bigr)\sigma  e_{1,3}\otimes e_{1,3},
\end{align*}
where $\sigma = \sigma(v, y)$ is the Kronecker function.
\end{example}


\medskip
\noindent
\emph{Proof of Theorem \ref{thm: semi-simple main formula}}. We divide the proof into several steps.

\medskip
\noindent
\textbf{Computation of a basis of $\Sol$}.
First,  we compute a basis of the vector space
\begin{equation}
 \Sol_{J, \, v, \, 0,\, \tau} :=
\left\{\Phi: \CC \to \Mat_{m \times n}(\CC)
\left|
\begin{array}{l}
\Phi \mbox{\rm{\quad is holomorphic}} \\
\Phi(z+1) = \Phi(z) \\
\Phi(z+\tau) J = e(z) J \Phi(z)
\end{array}
\right.
  \right\},
\end{equation}
where $e(z) = e(z, v, \tau)  = - \exp\bigl(-2\pi i (z + v + \tau)\bigr)$. The proof of the following result
is straightforward. 

\begin{lemma}\label{L:computeHoms}
Let $C: \CC \rightarrow \GL_{n}(\CC)$ and $D: \rightarrow \GL_{m}(\CC)$ be a pair of automorphy factors.
Let $\widetilde{D} := \bigl(D^{-1}\bigr)^{t}$ be the transpose of the inverse matrix
of $D$. Next, we set
$$
C \otimes \widetilde{D} =
\left(
\begin{array}{ccc}
c_{11} \widetilde{D} & \dots & c_{1n} \widetilde{D} \\
\vdots & \ddots & \vdots \\
c_{n1} \widetilde{D} & \dots & c_{nn} \widetilde{D}
\end{array}
\right).
$$
\begin{enumerate}
\item 
In   the notations of Theorem \ref{T:AutomFactors},  we have  isomorphisms
$$
\Hom\bigl(\kE(C), \, \kE(D)\bigr) \stackrel{\cong}\lar \Sol_{C, \,  D} 
\stackrel{\alpha}\lar \Sol_{(1), \, C \otimes \widetilde{D}}
\stackrel{\cong}\lar  H^0\bigl(\kE(C \otimes \widetilde{D})\bigr),
$$
where for  $\Phi = \bigl(f_{ij}\bigr)_{1 \le i, j \le n} \in \Sol_{C, \,  D}$ we set 
$\alpha\bigl(\Phi\bigr) = \bigl(f_{n(i-1) + j}\bigr)_{1 \le i, j \le n}$. 
\item We have: $J \otimes \widetilde{J} = \exp(N)$, where $N$ is the matrix from Definition
\ref{D:defin-of-matrixN}.
\end{enumerate}
\end{lemma}


\noindent
The following result is due to Polishchuk and Zaslow \cite[Proposition 2]{PolishchukZaslow}.

\begin{proposition}\label{L:PolishchukZaslow}
As above, let  $\nabla = -\frac{1}{2 \pi i} \cdot \frac{d}{dz}$
and $e(z) = - \exp\bigl(-2\pi i (z + v + \tau)\bigr)$.
Then we have an isomorphism of vector spaces:
$$
\delta: H^0\Bigl(\kE\bigl(e(z)\bigr)\Bigr) \otimes \CC^{n^2} \lar
H^0\Bigl(\kE\bigl(e(z) \cdot \exp(N)\bigr)\Bigr)
$$
given by the rule
$
\delta(f \otimes u) = \bigl(\exp(\nabla N) f\bigr) u =
\sum_{m = 0}^\infty \frac{\nabla^m(f)}{m!} N^m(u)
$
for any $f \in H^0\Bigl(\kE\bigl(e(z)\bigr)\Bigr)$ and $u \in \CC^{n^2}$.
\end{proposition}

\begin{proposition}\label{P:comput-of-Sol}
Let $\bar{\theta}_v(z) = \bar{\theta}(z + v + \frac{\tau+1}{2})$. Then
we have an isomorphism of vector spaces
$\Delta: \CC^{n^2} \rightarrow \Sol_{J,\, v, \, 0,\, \tau}$ mapping  a vector
$u \in \CC^{n^2}$ to the matrix-valued function  $\Delta(u)$, where for any $1 \le k, l \le n$
we have:
$$
\bigl(\Delta(u)\bigr)_{k, \, l}(z) =
e^t_{n(k-1)+l} \bigl(\exp(\nabla N) \bar{\theta}_v(z)\bigr) u.
$$
 \end{proposition}

\begin{proof}
By Lemma \ref{L:basic-for-theta},  the  vector space $H^0\bigl(\kE(e(z))\bigr)$ is one--dimensional and
$\bar{\theta}_v(z) = \bar{\theta}(z + v + \frac{\tau+1}{2})$ is its basis element. Hence, 
 Proposition  \ref{P:comput-of-Sol} is a consequence of 
 Lemma \ref{L:computeHoms}  and Proposition
\ref{L:PolishchukZaslow}.
\end{proof}

\begin{definition}
In the notations of Proposition \ref{P:comput-of-Sol}, let $U$ be the element
of $\Sol = \Sol_{J, \, v, \, 0, \, \tau}$ corresponding to  $u = e_{n(n-1)+1} \in \CC^{n^2}$. 
Note that $\bigl(U(z)\bigr)_{n, 1} = \bar{\theta}_v(z)$.
\end{definition}

\begin{proposition}\label{P:basis-of-Sol} Let $K = J_n(0)$ be the Jordan block of size $n \times n$ with eigenvalue zero.
For all $1 \le i, j \le n$ we set
$
F_{ij} = K^{n-i} U K^{j-1}.
$
Then we have:
\begin{enumerate}
\item\label{It:it1} All matrix--valued functions $F_{i j}: \CC \rightarrow \Mat_{n \times n}(\CC)$ belong to $\Sol$.
\item\label{It:it2} If $1 \le p, \, q \le n$ are such that  $i < p \le n$ or  $1 \le q < j$ 
then we have:
$\bigl(F_{i j }\bigr)_{p, \, q} = 0$. Moreover,
$\bigl(F_{i j }\bigr)_{i, \, j} = \bar{\theta}_v$. In other words, all non-zero entries
of $F_{ij}$ are located in  the rectangle  whose lower left corner is $(i, j)$.
\item\label{It:it3} Moreover, $\bigl\{F_{i j}\bigr\}_{1 \le i, j \le n}$ is a basis of the vector space $\Sol$.
\end{enumerate}
\end{proposition}

\begin{proof} The statement  that  $F_{i j}$ belongs to $\Sol$ is equivalent to the equality
\begin{equation}\label{E:onBasis-of-Sol}
K^{n-i} \, U(z+\tau) \, K^{j-1} \, J = e(z) J \, K^{n-i} \, U(z) \, K^{j-1}.
\end{equation}
Since the matrices $K$ and $J$ commute, Equality (\ref{E:onBasis-of-Sol}) is equivalent to
$$
K^{n-i} \bigl(U(z+\tau)J - e(z) J U(z)\bigr) K^{j-1} = 0,
$$
which is true since $U$ belongs to $\Sol$.
The second part of the proposition   follows from the definition of the functions $F_{ij}$.
From  this part  also follows    that all  elements of  the set $\bigl\{F_{i j}\bigr\}_{1 \le i, j \le n}$
are linearly independent. By Corollary \ref{C:dim-of-Sol}, the dimension of $\Sol$ is $n^2$.
Thus, $\bigl\{F_{i j}\bigr\}_{1 \le i, j \le n}$ is a basis of $\Sol$.
\end{proof}

\begin{example} Let $n = 2$.
Similarly to Example \ref{exa: M(2,0) ass. r-Matrix},
we obtain:
 \[
F_{2, 1} = U =\left(\begin{array}{cc}
\nabla\bar{\theta}_{v} & -\nabla^{2}\bar{\theta}_{v}\\
\bar{\theta}_{v} & -\nabla\bar{\theta}_{v}
\end{array}\right).
\]
 Moreover, we have:
 \[
F_{1,1} = \left(\begin{array}{cc}
\bar{\theta}_{v}  & -\nabla\bar{\theta}_{v} \\
0 & 0\end{array}\right),\,
F_{2,2} =
\left(\begin{array}{cc}  0 & \nabla\bar{\theta}_{v} \\
0 & \bar{\theta}_{v}
\end{array}\right) \; \mbox{and} \;
F_{1,2}=\left(\begin{array}{cc} 0 & \bar{\theta}_{v}\\
0 & 0\end{array}\right).\]
\end{example}

\medskip
\noindent
\textbf{Computation of $\res^{-1}_0$}.
As the next step, we compute the preimages of the elementary matrices 
$\bigl\{e_{a, b}\bigr\}_{1 \le a, b \le n}$ under the isomorphism
$\res_0: \Sol \rightarrow \Mat_{n \times n}(\CC)$.

\medskip
\noindent
Let $X=\left(x_{p, q}\right)_{1\leq p, q\leq n} \in \Mat_{n \times n}(\CC)$ be a
given matrix and $A \in \Sol$ be such that $\res_0(A) = X$.
By Proposition \ref{P:basis-of-Sol},  we have an expansion
$
A =\sum_{1\leq i,\,  j \leq n }\eta_{i,j}F_{i,j}
$
for certain uniquely determined  $\eta_{i,j}\in\mathbb{C}$. It is clear that  for all $1 \le p, q \le n$
 we get:
\begin{equation}\label{E:recurs-formula}
x_{p, q}=\sum_{\begin{smallmatrix}p\leq i\leq n\\
1\leq j\leq q\end{smallmatrix}}\eta_{i, j}\bigl(F_{i, j}(0)\bigr)_{p, q}.
\end{equation}
Next,  for all $1 \le p, q \le n$ we have
$\bigl(F_{p, q}(z)\bigr)_{p, q}=\bar{\theta}_{v}(z)$.
 This implies that
 \begin{equation}\label{Eq: first eta_p_q formula}
\eta_{p, q}=\frac{1}{\bar{\theta}_{v}(0)}\left[x_{p,q}-\sum_{\begin{smallmatrix}p\leq i\leq n\\
1\leq j\leq q\\
(i,j)\neq(p,q)\end{smallmatrix}}\eta_{i,j}\,\bigl(F_{i,j}(0)\bigr)_{_{p,q}}\right].
\end{equation}
Hence $\eta_{p,q}$ can be expressed as a linear combination of those
 $x_{i,j}$ for which  $p\leq i\leq n$
and $1\leq i\leq q$. Moreover, due to the recursive structure of the Equality 
(\ref{Eq: first eta_p_q formula}), it is clear that   $\eta_{p, \, q}$ can be written as a certain 
linear combination of $x_{i, j}$, whose
structure is  controlled  by the set of paths starting at  $(p,q)$ and ending at 
 $(i,j)$. In order to make this more precise, let us make the following
definition.

\begin{definition}
\label{def:W^t_(p,q),(i,j)} For any $\chi\in\mathbb{N}$ and $(i,j),(p,q)\in\mathbb{N}\times\mathbb{N}$
such that $i \ge p$ and $j \le q$ we denote
\[
\mathcal{W}_{(p,q),(i,j)}^{\chi}=\left\{ \left(\alpha_{s},\beta_{s}\right)_{0\leq s\leq\chi}\in(\mathbb{N}\times\mathbb{N})^{\chi+1}\left|\begin{smallmatrix}\alpha_{s}\leq\alpha_{s+1},\,\beta_{s}\geq\beta_{s+1}\\
\left(\alpha_{s},\beta_{s}\right)\neq\left(\alpha_{s+1},\beta_{s+1}\right)\\
\left(\alpha_{0},\beta_{0}\right)=\left(p,q\right)\\
\left(\alpha_{\chi},\beta_{\chi}\right)=\left(i,j\right)\end{smallmatrix}\right.\right\}. \]
In other words, $\mathcal{W}_{(p,q),(i,j)}^{\chi}$ is the set of all paths of length $\chi$ 
on the square lattice
$\mathbb{N} \times \mathbb{N}$ starting at $(p, q)$, ending at $(i, j)$ and going in the 
``south-west'' direction. 
\end{definition}

\noindent
Applying the recursive formula (\ref{Eq: first eta_p_q formula}), we end up with the following result.

\begin{lemma} 
Let $X=\bigl(x_{i, j}\bigr)_{1\leq i, j\leq n} \in \Mat_{n \times n}(\CC)$ and 
$\bigl\{\eta_{p, q}(X)\bigr\}_{1\leq p, q\leq n}$ be such that 
the equality
(\ref{E:recurs-formula}) is true. Then we have:
 \begin{equation}\label{eq: eta final formula}
\eta_{p,q}(X)=\sum_{\begin{smallmatrix}p\leq i\leq n\\
1\leq j\leq q\end{smallmatrix}}x_{i,j}\,
\left[\sum_{\chi=0}^{i-p+q-j}
\sum_{\mathcal{W}_{(p,q),(i,j)}^{\chi}}\frac{(-1)^{\chi}}{\bar{\theta}_{v}(0){}^{\chi+1}}
\prod_{s=0}^{\chi-1}\left(F_{\alpha_{s+1},\beta_{s+1}}(0)\right)_{\alpha_{s},\beta_{s}}\right],
\end{equation}
where the third sum runs over all elements $\bigl(\alpha_s, \beta_s\bigr)_{0 \le s \le \chi}$
of $\mathcal{W}_{(p,q),(i,j)}^{\chi}$.
\end{lemma}

\begin{corollary} Let $\bigl\{e_{a, b}\bigr\}_{1 \le a, b \le n}$ be the standard basis
of $\Mat_{n \times n}(\CC)$. If $a \ge p$ and $b \le q$ then   we have:
\[
\eta_{p,q}\left(e_{a,b}\right)=\sum_{\chi=0}^{a-p+q-b}
\sum_{(\alpha_s, \beta_s)_{0 \le s \le \chi} \in \mathcal{W}_{(p,q),(a,b)}^{\chi}}
\frac{(-1)^{\chi}}{\bar{\theta}_{v}(0){}^{\chi+1}}
\prod_{s=0}^{\chi-1}\left(F_{\alpha_{s+1},\beta_{s+1}}(0)\right)_{\alpha_{s},\beta_{s}},
\]
whereas in the remaining cases  $\eta_{p,q}(e_{a,b})=0$. Hence,
$\res_{0}^{-1}\left(e_{a,b}\right)=\gamma^{a,b}$,
where 
\[
\gamma^{a,b}(z) = 
\sum_{\begin{smallmatrix}1\leq p\leq a\\
b\leq q\leq n\end{smallmatrix}}F_{p,q}(z)\,\left[\sum_{\chi=0}^{a-p+q-b}
\sum_{(\alpha_s, \beta_s)_{0 \le s \le \chi} \in \mathcal{W}_{(p,q),(a,b)}^{\chi}}\frac{(-1)^{\chi}}{\bar{\theta}_{v}(0){}^{\chi+1}}
\prod_{s=0}^{\chi-1}\left(F_{\alpha_{s+1},\beta_{s+1}}(0)\right)_{\alpha_{s},\beta_{s}}\right].
\]
\end{corollary}

\noindent
Denote  the matrix entries of $\left(\gamma^{a,b}(z)\right)_{c,d}$
by $\gamma_{c,d}^{a,b}(z)$. If $c < a$ or $d < b$ then $\gamma_{c,d}^{a,b}(z) = 0$.
On the other hand, if $a \ge c$ and $d \ge b$ then we get:
 \begin{equation}\label{eq:gamma formula 1}
\gamma_{c,d}^{a,b}(z)=\sum_{\begin{smallmatrix}c\leq p\leq a\\
b\leq q\leq d\end{smallmatrix}}\Bigl[\bigl(F_{p,q}(z)\bigr)_{c,d}
\sum_{\chi=0}^{a-p+q-b}
\sum_{\mathcal{W}_{(p,q),(a,b)}^{\chi}}\frac{(-1)^{\chi}}{\bar{\theta}_{v}(0){}^{\chi+1}}
\prod_{s=0}^{\chi-1}\left(F_{\alpha_{s+1},\beta_{s+1}}(0)\right)_{\alpha_{s},\beta_{s}}\Bigr].
\end{equation}

\begin{remark}
In the formula (\ref{eq:gamma formula 1}) the function  $\gamma_{c,d}^{a,b}$ depends on 
one variable $z$. However, from its definition it is clear that it also depends on the parameter
$v \in \CC \setminus \Lambda$. Hence, in what follows, we shall consider it as a function
of two variables $z$ and $v$. 
\end{remark}

\noindent
\textbf{Computation of $\tilde{r}_{J}(v, y) :\Mat_{n\times n}(\mathbb{C})\rightarrow \Mat_{n\times n}(\mathbb{C})$}.
Recall that $\tilde{r}_{B}(v,y): {\Mat}_{n\times n}(\mathbb{C})\rightarrow {\Mat}_{n\times n}(\mathbb{C})$
is the composition $\ev_y \circ \res^{-1}_0$.
The tensor  $r_{B}(v,y) \in \Mat_{n \times n}(\CC) \otimes \Mat_{n \times n}(\CC)$
 is the image of  $\tilde{r}_{B}(v,y)$
under  the canonical isomorphism \[
\Lin\bigl({\Mat}_{n\times n}(\mathbb{C}), {\Mat}_{n\times n}(\mathbb{C})\bigr)\lar
 {\Mat}_{n\times n}(\mathbb{C})\otimes {\Mat}_{n\times n}(\mathbb{C}).
 \]
 In the standard basis $\bigl\{e_{a, b}\bigr\}_{1 \le a, b \le n}$ of $\Mat_{n \times n}(\CC)$ 
 this map
 is given as follows:
\[
\left(e_{a,b}\mapsto\sum_{1\leq c, \, d\leq n}\alpha_{c,d}^{a,b}\, e_{c,d}\right)\mapsto
\sum_{1\leq c, \, d\leq n}\alpha_{c,d}^{a,b}\, e_{b,a}\otimes e_{c,d}.
\]
Hence, the solution of the associative Yang--Baxter equation (\ref{E:AYBE}) attached 
to the Jordan block $J$ is the following:
 \begin{eqnarray}
r_{J}(v,y) & =\frac{1}{\bar{\theta}\left(\frac{1+\tau}{2}+y\right)} & \sum_{\begin{smallmatrix}0\leq k\leq n-1\\
0\leq l\leq n-1\end{smallmatrix}}\sum_{\begin{smallmatrix}1\leq i\leq n-l\\
1\leq j\leq n-k\end{smallmatrix}}\gamma_{j,i+l}^{j+k,i}(v, y)\, e_{i,j+k}\otimes e_{j,i+l},\label{eq: first formula for r_B}\end{eqnarray}
where $\gamma_{j,i+l}^{j+k,i}(v, y)$ are given by  (\ref{eq:gamma formula 1}).
Our next goal is to simplify this expression.

\begin{definition}
For any $i,j,\chi$ in $\mathbb{N}$, we set
 \[
\mathcal{W}_{(i,j)}^{\chi}=\left\{ \left(\alpha_{s},\beta_{s}\right)_{1\leq s\leq\chi}\in\left(\mathbb{N}\times\mathbb{N}\right)^{\chi}\left|
(\alpha_{s},\beta_{s})\neq(0,0); \; 
\left(\sum_{s=1}^{\chi}\alpha_{s},\sum_{s=1}^{\chi}\beta_{s}\right)=(i,j)
\right.
\right\}. \]
We
may think of $\mathcal{W}_{(i,j)}^{\chi}$ as the set paths of length
$\chi$ in $\mathbb{N}\times\mathbb{N}$ starting at the point  $(0,0)$,
ending in $(i,j)$ and going in the ``nord-east'' direction. Note that any path
$\bigl\{(\alpha_{s},\beta_{s})\bigr\}_{0 \le s \le \chi}$ from $\mathcal{W}_{(j+a,i+b),(j+k,i)}^{\chi}$
corresponds to   the element $\bigl\{(\alpha_{s+1}-\alpha_{s},\beta_{s}-\beta_{s+1})_{s}\bigr\}_{1 \le s \le \chi}$ of
 $\mathcal{W}_{(k-a,b)}^{\chi}$. For the sake of simplicity, we shall use the same 
 notation $\bigl\{(\alpha_{s},\beta_{s})\bigr\}_{0 \le s \le \chi}$
 for both elements.
\end{definition}

\noindent
In these notations, the  formula (\ref{eq:gamma formula 1}) can be rewritten as follows:
 \[
\gamma_{j,i+l}^{j+k,i}(y)=\sum_{\begin{smallmatrix}0\leq a\leq k\\
0\leq b\leq l\end{smallmatrix}}\left[\bigl(F_{j+a,i+b}(y)\bigr)_{j,i+l}\cdot \sum_{\chi=0}^{k-a+b}\sum_{\mathcal{W}_{(k-a,b)}^{\chi}}\frac{(-1)^{\chi}}{\bar{\theta}_{v}(0){}^{\chi+1}}
\prod_{s=1}^{\chi}\left(F_{\alpha_{s+1},\beta_{s+1}}(0)\right)_{\alpha_{s},\beta_{s}}\right],
\]
where the third sum is taken over all elements $\bigl(\alpha_s, \beta_s\bigr)_{1 \le s \le \chi}$
of $\mathcal{W}_{(k-a,b)}^{\chi}$.
Recall
that 
$\bigl(F_{\alpha,\beta}\bigr)_{\gamma, \delta} = 0$ if $\gamma > \alpha$ or $\beta > \delta$. 
For $\alpha \ge \gamma$ and $\delta \ge \beta$ we have: 
$
\bigl(F_{\alpha,\beta}\bigr){}_{\gamma, \delta} =
U_{(n-(\alpha-\gamma), \delta-\beta+1)},
$
where $
\bigl(U(v, z)\bigr)_{\alpha,\beta}=e_{n(\alpha-1)+\beta}^{t}\, \exp(\nabla\, N)\bar{\theta}_{v}(z)\, e_{n(n-1)+1}$.
Hence, 
 \[
\gamma_{j,i+l}^{j+k,i}(v, y)=\sum_{\begin{smallmatrix}0\leq a\leq k\\
0\leq b\leq l\end{smallmatrix}}
\left[
\bigl(U(v, y)\bigr){}_{(n-a,l-b+1)}\cdot
\sum_{\chi=0}^{k-a+b}\sum_{\mathcal{W}_{(k-a,b)}^{\chi}}\frac{(-1)^{\chi}}{\bar{\theta}_{v}(0){}^{\chi+1}}
\prod_{s=1}^{\chi}\bigl(U(0)\bigr)_{n-\alpha_{s},\beta_{s}+1}
\right]
\]
 Next, we need the following combinatorial lemma.

\begin{lemma} \label{lem: powers of N} In the notations of Definition \ref{D:defin-of-matrixN}
 we have the following results.
\begin{enumerate}
\item\label{It:part1} For all $r\in\mathbb{N}$ we get \[
N^{r}=
\left(\begin{array}{cccc}
N_{0}^{(r)} & N_{1}^{(r)} & \cdots & N_{n-1}^{(r)}\\
0 & \ddots & \ddots & \vdots\\
\vdots & \ddots & \ddots & N_{1}^{(r)}\\
0 & \cdots & 0 & N_{0}^{(r)}
\end{array}
\right),\]
where each $N_{i}^{(r)}$, $0\leq i\leq n-1$ is the matrix of size $n \times n$ 
given by\[
N_{i}^{(r)}=\sum_{k=0}^{r}(\begin{smallmatrix}r\\
k\end{smallmatrix})\, A_{0}^{r-k}\sum_{s\in S_{i}^{k}}\prod_{l=0}^{k-1}A_{s_{l}-s_{l+1}}\]
where for all $k,i\in\mathbb{N}$ we denote 
\[
S_{i}^{k}=\left\{ \left(s_{j}\right)_{0\leq j\leq k}\in\mathbb{N}^{k+1}\big|
s_{0}=i, 
s_{k}=0, 
s_{j+1}<s_{j},\,0\leq j\leq k
\right\}.
 \]

\item For any $i,j,\chi\in\mathbb{N}$, $\left(u_{s}\right)_{1\leq s\leq\chi}\in[0,2n-1]^{\chi}$
and $u=\sum_{s=1}^{\chi}u_{s}$ holds:
 \[
\sum\limits_{
(\alpha_s, \beta_s)_{1 \le s \le \chi} \in \mathcal{W}_{(i,j)}^{\chi}}
\prod_{s=1}^{\chi}e_{\beta_{s}+1}^{t}\, N_{\alpha_{s}}^{\left(u_{s}\right)}\, e_{1}=e_{j+1}^{t}\, N_{i}^{\left(u\right)}\, e_{1}.
\]
\end{enumerate}
\end{lemma}

\noindent
The  proof of this lemma is based  on some
  elementary but tedious
combinatorics, and  is therefore omitted.

\medskip
\noindent
From Lemma \ref{lem: powers of N}(\ref{It:part1}) it follows that for all 
$0 \le \alpha < n$, $1 \le \beta \le n$
we have:
\[
\bigl(U(v, z)\bigr)_{n-\alpha,\beta}=e_{n(n-\alpha-1)+\beta}^{t}\, \exp(\nabla_z\, N)\bar{\theta}_{v}(z)\, e_{n(n-1)+1}
=\sum_{r=0}^{2n-1}e_{\beta}^{t}\, N_{\alpha}^{(r)}\, e_{1}\,
\frac{\nabla_z^r\bigl(\bar{\theta}_v(z)\bigr)}{r!}.
\]
Recall that $\bar{\theta}_v(z) = \bar{\theta}(z + v + \frac{1+\tau}{2})$. Note that we have:
$$
\nabla_z^r\bigl(\bar{\theta}_v(z)\bigr) = \nabla_v^r\bigl(\bar{\theta}_v(z)\bigr) =
\left(\frac{\partial}{\partial v}\right)^{r}\bar{\theta}\left(z+\frac{\tau+1}{2}+v\right).
$$
Therefore, we can rewrite the expression for $\gamma_{j,i+l}^{j+k,i}(v, y)$ as follows:
\begin{equation}\label{eq: gamma before faa da bruni lemma}
\gamma_{j,i+l}^{j+k,i}(v, y) = 
\sum_{\begin{smallmatrix}0\leq a\leq k\\
0\leq b\leq l\end{smallmatrix}}\left(\left(\sum_{r=0}^{2n-1}e_{l-b+1}^{t}\, N_{a}^{(r)}\, e_{1}\,\frac{\nabla_v^r \bigl(\bar{\theta}_v(y)\bigr)}{r!}\right)\cdot\right. 
\end{equation}
$$
\cdot\biggl.
\sum_{\chi=0}^{k-a+b}
\sum_{ (\alpha_s, \beta_s)_{1 \le s \le \chi} \in \mathcal{W}_{(k-a,b)}^{\chi}}\frac{(-1)^{\chi}}{\bar{\theta}_{v}(0){}^{\chi+1}}
\prod_{s=1}^{\chi}\biggl(\sum_{r_{s}=0}^{2n-1}e_{\beta_{s}+1}^{t}\, N_{\alpha_{s}}^{(r_{s})}\, e_{1}\,\frac{\nabla_v^{r_s}\bigl(\bar{\theta}_v(0)\bigr)}{r_{s}!}\biggr)
\biggr).
$$
Next, we  need the following generalization of the Leibniz formula.
\begin{lemma}\label{L:heavy-exercise}
Let $f,g$ be any meromorphic functions on $\mathbb{C}$ and  $\nabla = -\frac{1}{2\pi i} \frac{d}{dz}$.
Then  in the notations of Definition \ref{D:def-of-nabla} the following formula is true:
\begin{align*}
\nabla_{k,l}\Bigl(\frac{f}{g}\Bigr) = 
\sum_{
\begin{smallmatrix}0\leq a\leq k\\
0\leq b\leq l
\end{smallmatrix}}
\Bigl[
\Bigl(\sum_{r=0}^{2n-1}e_{l-b+1}^{t}\, N_{a}^{(r)}\, e_{1}\,\frac{\nabla^r(f)}{r!}
\Bigr)
\cdot
\Bigl(\sum_{\chi=0}^{k-a+b}
\sum_{(\alpha_s, \beta_s)_{1 \le s \le \chi} \in \mathcal{W}_{(k-a,b)}^{\chi}}
\frac{(-1)^{\chi}}{g^{\chi+1}} \cdot \\
\cdot \prod_{s=1}^{\chi}\sum_{r_{s}=0}^{2n-1}e_{\beta_{s}+1}^{t}\, N_{\alpha_{s}}^{(r_{s})}\, e_{1}\,\frac{\nabla^{r_s}(g)}{r_{s}!}
\Bigr)
\Bigr].
\end{align*}
\end{lemma}

\noindent
The proof of this lemma is based on  some  elementary combinatorics
and  the second part of Lemma \ref{lem: powers of N}. We leave it to an interested reader as an exercise. 
\qed 

\medskip
\noindent
Applying Lemma \ref{L:heavy-exercise} to Equality  (\ref{eq: gamma before faa da bruni lemma}), we finally
get:
\[
\frac{1}{\bar{\theta}\left(\frac{1+\tau}{2}+y\right)}
\gamma_{j,i+l}^{j+k,i}(v, y)=
\bigl(\nabla_{k,l}\bigr)_v
\left(\frac{\bar{\theta}(y+\frac{\tau+1}{2}+v)}{\bar{\theta}\left(\frac{1+\tau}{2}+y\right) \cdot \bar{\theta}(\frac{\tau+1}{2}+v)}\right).
\]
Recall that 
the first and third 
theta  functions $\theta$ and  $\bar{\theta}$  are related
by the equality 
$\bar{\theta}\left(z+\frac{1+\tau}{2}\right)=i\, q(z)\,\theta(z)$,
where $q(z)=\mbox{exp}(-\pi i(z+\frac{\tau}{4}))$. 
Thus, up to the constant  $\frac{\exp(\pi i \frac{\tau}{4})}{i \theta'(0)}$, 
the coefficient of the tensor $e_{i,j+k}\otimes e_{j,i+l}$ in the expansion 
(\ref{eq: first formula for r_B}) is $(\nabla_{k,l})_v\bigl(\sigma(v,y)\bigr)$.
This finishes the proof of Theorem \ref{thm: semi-simple main formula}.

 \begin{remark}
 The algorithm from Section \ref{S:BundlesandAYBE}
 assigning  to a matrix $B\in \GL_n(\CC)$ and a complex torus $E$ a solution
 of the associative  Yang--Baxter can be generalized to the case when $E$ is a singular Weierstra\ss{}
 cubic curve. In this case, one can use  a description of semi--stable vector bundles on $E$
 following the approach of \cite{Survey}, see also \cite{BK4}. However, all solutions produced in this way turn out to be degenerations of the constructed elliptic solutions, where we replace the Kronecker function
 $\sigma(u, x)$ by its trigonometric or rational degenerations
 $\cot(u) + \cot(x)$ or $\displaystyle{\frac{1}{u} + \frac{1}{x}}$.
 \end{remark}

\end{document}